\documentclass[11pt]{amsart}
\usepackage{todonotes}
\usepackage[margin=1in]{geometry}
\usepackage{amsmath,amssymb,amsthm,mathtools}
\usepackage{hyperref}
\usepackage[shortlabels]{enumitem}
\usepackage{xcolor}

\usepackage[no-math]{fontspec}
\usepackage[frozencache=true,cachedir=minted-cache]{minted}

\newmintinline[lean]{lean4}{bgcolor=white}
\newminted[leancode]{lean4}{escapeinside=!!, breaklines, fontsize=\normalsize}
\newtheorem{theorem}{Theorem}[section]
\newtheorem{proposition}[theorem]{Proposition}
\newtheorem{lemma}[theorem]{Lemma}
\newtheorem{corollary}[theorem]{Corollary}
\newtheorem{definition}[theorem]{Definition}
\theoremstyle{remark}
\newtheorem{remark}[theorem]{Remark}

\newcommand{\C}{\mathbb{C}}
\newcommand{\N}{\mathbb{N}}
\newcommand{\abs}[1]{\lvert #1 \rvert}
\newcommand{\norm}[1]{\lVert #1 \rVert}

\newcommand{\dd}{\mathrm{d}}
\newcommand{\ee}{\mathrm{e}}
\newcommand{\ii}{\mathrm{i}}
\newcommand{\rhoFn}{\rho}
\newcommand{\Fd}{\mathcal{F}_d}
\newcommand{\Sphere}{S^{2d-1}}

\title{$L^2$-Stability for STFT phase retrieval}

\author[S.~Bertolini]{Susanna Bertolini}
    \address{Department of Mathematics\\
             ETH Z\"urich, Ramistrasse 101, 8092 Z\"urich, Switzerland.}
    \email{susanna.bertolini@math.ethz.ch}
\author[J.~de Dios Pont]{Jaume de Dios Pont}
    \address{Center for Data Science,
             New York University, 
             New York, New York 10011, USA}
    \email{jdedios@nyu.edu}
\author[B.~Pineau]{Ben Pineau}
    \address{Courant Institute for Mathematical Sciences\\
             New York University} 
     \email{brp305@nyu.edu}
\author[J.~Ramos]{Jo\~ao~P.~G.\ Ramos}
    \address{Instituto de Matem\'atica Pura e Aplicada (IMPA), 
             Rio de Janeiro, RJ, Brazil}
    \email{joao.ramos@impa.br}
\author[M.~Taylor]{Mitchell~A.\ Taylor}
    \address{Department of Mathematics\\
             ETH Z\"urich, Ramistrasse 101, 8092 Z\"urich, Switzerland.}
    \email{mitchell.taylor@math.ethz.ch}

\begin{document}

\begin{abstract}
We prove that the short-time Fourier transform with Gaussian window performs $L^2$-local stable phase retrieval at the constant function. The proof involved significant interplay between mathematicians and LLMs. An autoformalization in Lean 4 of an extension of our result to $L^2$-local stable phase retrieval for all Hermite windows and all elements in the finite span of the canonical basis vectors is also presented.
\end{abstract}

\maketitle

\section{Introduction}
Short-time Fourier transform (STFT) phase retrieval refers to the problem of recovering an unknown function $f\in L^2(\mathbb{R}^d)$ from the phaseless measurement $|V_gf|$, where $V_g:L^2(\mathbb{R}^d)\to L^2(\mathbb{R}^{2d})$ denotes the STFT with window function $g\in L^2(\mathbb{R}^d)$. This problem arises in ptychography, which is a widely used imaging technique that obtains information about an unknown object $f$ by illuminating it with shifts of a probe $g$. 
Due to the high oscillation of the diffracted waves, practitioners can only record the far-field intensity $|V_gf|^2$, which leads to the phase retrieval problem described above.
 \vspace{0.5em}

The mathematics of STFT phase retrieval is a rich subject, with contributions coming from many different communities. Recall that, given a window function $g\in L^2(\mathbb{R}^d)$, the \emph{short-time Fourier transform} $V_g:L^2(\mathbb{R}^d)\to L^2(\mathbb{R}^{2d})$ is defined by
\begin{equation*}\label{STFT Def}
V_g f(x,\omega)=\mathcal{F}(T_x\overline{g}f) (\omega)=\int_{\mathbb{R}^d} f(t)\overline{g(t-x)}e^{-2\pi it\cdot \omega}dt, \ \ \text{for}\ x,\omega\in \mathbb{R}^d.
\end{equation*}
It is easy to see that 
\begin{equation}\label{isometry}
    \|V_g f\|_{L^2(\mathbb{R}^{2d})}=\|g\|_{L^2(\mathbb{R}^{d})} \|f\|_{L^2(\mathbb{R}^{d})}
\end{equation} and the fundamental problem in STFT phase retrieval is to stably reconstruct $f\in L^2(\mathbb{R}^d)$ from the phaseless measurement $|V_g f|$, up to global phase ambiguity.  \vspace{0.5em}

If stability issues are ignored, there is a standard method for determining whether the operator $V_g$ does phase retrieval. Recall that the \emph{ambiguity function} $\mathcal{A}f$  of a square-integrable function $f$ is defined by $\mathcal{A}f(x,\omega):=e^{\pi ix\cdot \omega} V_ff(x,\omega)$. It is well-known that $\mathcal{A}f$ is a continuous function that determines every $f\in L^2(\mathbb{R})$ up to global phase. Moreover, we have the relation 
\begin{equation}\label{Ambiguity to STFT}
    \mathcal{F}\left( |V_g f|^2\right)(\omega,-x)=\mathcal{A}f(x,\omega) \overline{\mathcal{A}g(x,\omega)}.
\end{equation}
For the Gaussian window $g$, it is easy to see that $\mathcal{A}g$ is again a Gaussian, so the above identity implies that $V_g$ does phase retrieval. Similarly, since the ambiguity function of each Hermite window only vanishes on circles, these windows do STFT phase retrieval.
 \vspace{0.5em}

The stability problem for STFT phase retrieval has been studied by many authors \cite{alaifari2019Stable,alaifari2025cheeger,fuhr2025cheeger,MR4709358,Grohs2019Stable,MR4404785,MR5010357}. The most natural notion of local stability is as follows.
\begin{definition}\label{defn of local stability}
    Fix a window function $g\in L^2(\mathbb{R}^d)$. We say that $V_g$ does \emph{local stable phase retrieval at $f\in L^2(\mathbb{R}^d)$} if there exists a constant $C(f)<\infty$ such that for all $h\in L^2(\mathbb{R}^d)$ we have
    \begin{equation}\label{SPR local V_g}
        \inf_{|\lambda|=1}\|f-\lambda h\|_{L^2(\mathbb{R}^d)}\leq C(f) \| |V_g f|-|V_g h|\|_{L^2(\mathbb{R}^{2d})}.
    \end{equation}
\end{definition}
Despite extensive efforts on this problem, we are unaware of a single pair of functions $f,g\in L^2(\mathbb{R}^d)\setminus\{0\}$ for which $V_g$ does local stable phase retrieval at $f$. On the other hand, for \emph{any} window $g\in L^2(\mathbb{R}^d)\setminus\{0\}$, it is known \cite{alaifari2025cheeger,Alharbi2024Locality} that one can construct a dense set of functions $f\in L^2(\mathbb{R}^d)$ for which $C(f)=\infty$. 
 \vspace{0.5em}

The objective of the present article is to show that there are many functions $f,g\in L^2(\mathbb{R}^d)\setminus\{0\}$ for which $V_g$ does local stable phase retrieval at $f$ in the sense of Definition~\ref{defn of local stability}. In fact, when $g$ is the Gaussian or any other Hermite window, there is a  canonical dense set of functions with $C(f)<\infty.$

\subsection{History and preliminaries} Influential results of Cahill-Casazza-Daubechies \cite{MR3554699} and Alaifari and Grohs \cite{MR3656501} show that phase retrieval using frames  cannot be uniformly stable, unless the Hilbert space $\mathcal{H}$ is finite-dimensional. More precisely,  whenever $T:\mathcal{H} \to L^2(\mu)$ is the analysis operator of a frame and $\dim\mathcal{H}=\infty$, we have $\sup_{f\in \mathcal{H}}C(f)=\infty,$ where $C(f)\in [0,\infty]$ is the smallest constant satisfying
\begin{equation*}\label{LSPR intro}
    \inf_{|\lambda|=1}\|f-\lambda g\|_{\mathcal{H}}\leq C(f)\| |Tf|-|Tg|\|_{L^2(\mu)} \ \text{for all} \ g\in \mathcal{H}.
\end{equation*} 
 In fact, it was recently shown in \cite{Alharbi2024Locality} that we have $C(f)=\infty$ for a dense collection of $f\in \mathcal{H}$. In particular, in the natural setting of Definition~\ref{defn of local stability}, instabilities cannot be avoided. 
 \vspace{0.5em}

 In view of the negative results above, the question of \emph{how} instabilities occur became fundamental. Importantly, the instabilities found in the aforementioned works can be viewed as mathematical artifacts rather than inherent issues with the experimental setup. Indeed, they all occur (in some sense) ``at infinity". This intuition has been made precise in the celebrated works of Alaifari-Daubechies-Grohs-Yin \cite{alaifari2019Stable} and Grohs-Rathmair \cite{Grohs2019Stable,MR4404785}, and more recently in \cite{alaifari2025cheeger}. However, all of these papers have the significant drawback that they only apply to the Gaussian window and (for technical reasons) they must replace the natural $L^2$-norm on the right-hand side of \eqref{SPR local V_g} with a weighted first order Sobolev norm.
 \vspace{0.5em}
 
 The lack of robustness when the window is changed is obviously an important issue, as the window is not precisely known in practice. However, the issue of changing the topology is equally problematic, as it makes the forward problem $f\mapsto |V_gf|$ ill-posed. In view of this, it has been an important open problem to understand when stability holds in the natural $L^2$-norm and to what extent one may replace the Gaussian window by a more general window \cite{rathmair2024stable}.

\subsubsection{Reformulating the problem}
Let $g\in L^2(\mathbb{R}^d)$ satisfy $\|g\|_{ L^2(\mathbb{R}^d)}=1$. In view of the identity \eqref{isometry}, Definition~\ref{defn of local stability} can be seen as a special case of the following unified problem by taking $E=V_g(L^2(\mathbb{R}^{d}))\subseteq L^2(\mathbb{R}^{2d})$.

\begin{definition}\label{defn of local stability2}
    Let $E$ be a subspace of $L^2(\Omega,\mu)$ for some measure space $(\Omega,\mu)$. We say that $E$ does \emph{local stable phase retrieval at $f\in E$} if there exists a constant $C(f)<\infty$ such that for all $h\in E$ we have
    \begin{equation*}
        \inf_{|\lambda|=1}\|f-\lambda h\|_{L^2(\mu)}\leq C(f) \| |f|-|h|\|_{L^2(\mu)}.
    \end{equation*}
\end{definition}
In a pioneering work by Calderbank-Daubechies-Freeman-Freeman \cite{calderbank2022stable}, infinite-dimensional subspaces $E$ of real-valued $L^2(\mathbb{R}^d)$ have been constructed that satisfy $\sup_{f\in E}C(f)<\infty$. Similar examples over the complex field were later constructed in \cite{MR4837558} and a general theory of stable phase retrieval in Banach lattices was initiated in \cite{MR4800909}. This has led to a host of new and exciting results \cite{camunez2025characterization,garcia2025isometric,garcia2025existence} that lie in the intersection of various fields. When we state our main results below, we will use the language of Definition~\ref{defn of local stability2} rather than Definition~\ref{defn of local stability}.

\subsubsection{Function spaces} \label{Sec:Fun}

For any measurable function $F \colon \mathbb{C} \to \mathbb{C}$, we define the
\emph{Fock norm}
\[
\lVert F \rVert_{\mathcal{F}}^2
\;:=\lVert F \rVert_{\gamma}^2
\;:=\;
\frac{1}{\pi}\int_{\mathbb{C}} \lvert F(z) \rvert^2 \, \mathrm{e}^{-\lvert z \rvert^2}\,\mathrm{d} m(z).
\]
The \emph{Bargmann--Fock space}~$\mathcal{F}^2(\mathbb{C})$ is the closed subspace of entire functions with
$\lVert F \rVert_{\mathcal{F}} < \infty$.
The monomials $e_n(z) := z^n/\sqrt{n!}$ form an orthonormal basis of $\mathcal{F}^2(\mathbb{C})$, so every
$F(z) = \sum_{n \ge 0} a_n z^n$ satisfies
$\lVert F \rVert_{\mathcal{F}}^2 = \sum_{n \ge 0} \lvert a_n \rvert^2 n!$.
With this convention, the Fock norm is defined for arbitrary measurable
functions --  not only holomorphic ones -- and in particular we have
$\lVert \,\lvert F \rvert\, \rVert_{\mathcal{F}} = \lVert F \rVert_{\mathcal{F}}$.
 \vspace{0.5em}

Up to universal factors, one may identify the Fock space with the image of the short-time Fourier transform with Gaussian window.
Since elements of the Fock space are entire functions, it is obvious that phase retrieval holds in this subspace, i.e., for all $F,G\in \mathcal{F}^2(\mathbb{C})$ we have $|F|=|G|$ if and only if $F=\lambda G$ for some unimodular scalar $\lambda$. However, the question of stability is far more subtle. In fact, it was recently shown \cite{alaifari2025cheeger,Alharbi2024Locality} that there is a dense set of elements of $\mathcal{F}^2(\mathbb{C})$ that \emph{fail} local stable phase retrieval (even with derivative loss). The main goal of this article is to construct a complementary dense set that \emph{does} local stable phase retrieval in the natural $\mathcal{F}$-norm.
 \vspace{0.5em}

Our stability results  apply not only to the Gaussian window, but to any Hermite window. In this case (again up to universal factors), the STFT maps $L^2(\mathbb{R})$ onto the \emph{true polyanalytic Fock space}\footnote{The full polyanalytic Fock space fails phase retrieval whenever $k\geq 1$, since $|\bar{z}|=|z|$. See, for example, \cite{MR3203099} for a discussion on the true polyanalytic Fock space and its connections to time-frequency analysis.} $\mathcal{H}_k(\mathbb{C})$. Here, for any integer $k\geq 0$, we define $\mathcal{H}_k(\mathbb{C})$ as the $\|\cdot\|_\gamma$-closed linear span of the  basis
\begin{equation*}
    \Phi_{k,n}(z)=\frac{1}{\sqrt{k!n!}}\sum_{j=0}^{\min(k,n)}(-1)^j\binom{k}{j}\frac{n!}{(n-j)!}z^{n-j}\overline{z}^{k-j}, \ \ n\geq 0.
\end{equation*}
Clearly, $\mathcal{H}_0(\mathbb{C})=\mathcal{F}^2(\mathbb{C})$ is the Bargmann--Fock space and $\mathcal{H}_k(\mathbb{C})$ is a subspace of polyanalytic functions. By \eqref{Ambiguity to STFT}, $\mathcal{H}_k(\mathbb{C})$  performs phase retrieval. However, the major conceptual difference compared with the case $k=0$ is that the ambiguity functions of the higher Hermite windows have non-trivial zero sets. A priori, this could affect the stability of the phase recovery process, since solving for $\mathcal{A}f$ in \eqref{Ambiguity to STFT} would require division by the ambiguity function of the window.
\subsubsection{Higher dimensional generalizations} The above definitions lift naturally to higher dimensions. Indeed, the $d$-dimensional basis consists of tensor products of the one-dimensional $\Phi_{k,n}$. To make this precise, we fix a dimension $d\geq 1$ and a multi-index $\vec \kappa$. For each multi-index $\vec \alpha$, we define the function
\[
\Phi_{\vec\kappa,\vec\alpha}(\vec z)
:=\prod_{q=0}^{d-1}\Phi_{\vec\kappa_q,\vec\alpha_q}(\vec z_q), \hspace{5mm} \vec z=(z_0,\dots,z_{d-1}).
\]
The space $\mathcal{H}_{\vec\kappa}(\mathbb{C}^d)$ is then defined as the closure of  $\operatorname{span}\{\Phi_{\vec\kappa, \vec\alpha} :\vec \alpha \in \N^d\}$ in the norm
\begin{equation*}
  \|F\|_{\mathcal{F}_d}^2:= \|F\|_{\gamma_d}^2:=\frac{1}{\pi^d}\int_{\mathbb{C}^d}|F(\vec z)|^2  e^{-\lVert \vec z\rVert^2}dm(\vec z).
\end{equation*}

\subsection{Main results}
The main results of this paper concern the $L^2$-local stability of STFT phase retrieval for the canonical window functions at elements of the span of the canonical basis vectors. However, rather than stating one main result, we prefer to state the result in stages, following the path taken by the verification in the Lean 4 formal proof verification software \cite{lean4}. 
 \vspace{0.5em}

The main result for which we will present a detailed proof is the following.
\begin{theorem}[Local SPR at $\mathbf{1}$ in $\mathcal{F}^2(\mathbb{C})$]\label{main:local}
There exists an absolute constant $M> 0$ such
that for any $F\in \mathcal{F}^2(\mathbb{C})$ we have
\[
\inf_{|\lambda|=1}\lVert F - \lambda \rVert_{\mathcal{F}}
\;\le\;
M\,
\bigl\|\, \lvert F \rvert - 1 \,\bigr\|_{\mathcal{F}}.
\]
\end{theorem}

Theorem~\ref{main:local} follows from Theorem~\ref{thm:local}, a standard density argument, and a general fact about reproducing kernel Hilbert spaces that we will discuss below. 
\begin{theorem}[Local-local SPR at $\mathbf{1}$ in $\mathcal{F}^2(\mathbb{C})$]\label{thm:local}
There exist absolute constants $\delta > 0$ and $\widetilde{M} > 0$ such
that the following holds.  Let $P$ be a polynomial with
$\lVert P \rVert_{\mathcal{F}} \le \delta$.  Then there exists $w \in \mathbb{C}$ with $\lvert w \rvert = 1$
such that
\[
\lVert w(1+P) - 1 \rVert_{\mathcal{F}}
\;\le\;
\widetilde{M}\,
\bigl\|\, \lvert 1+P \rvert - 1 \,\bigr\|_{\mathcal{F}}.
\]
One may take $\delta = 1/4601$ and $\widetilde{M} = 23003$.
\end{theorem}
We remark that the assumption that $P$ is a polynomial in Theorem~\ref{thm:local} is for convenience, as a standard density argument allows one to pass from a polynomial $P$ to a general element $F\in \mathcal{F}^2(\mathbb{C})$. The reduction from Theorem~\ref{thm:local} to Theorem~\ref{main:local} is a general RKHS fact that does not rely on any further structure of the Fock space. We refer the reader to Section~\ref{subsec:local-reduction} for this general result.
 \vspace{0.5em}

Theorem~\ref{main:local} can be vastly generalized using similar methods. Since the proof becomes much more technical but not more instructive, we will only sketch the main changes and provide a link to the full Lean verification. Our most general result establishes local stability for every Hermite window at every element in the linear span of the canonical basis vectors.
\begin{theorem}[General stability for Hermite phase retrieval]\label{thm:local_Gen}
Fix $d\geq 1$ and a multi-index $\vec\kappa$. Let $P_0\in \text{span} \{\Phi_{\vec \kappa,\vec\alpha}\}_{\vec \alpha}$ be an element of the finite linear span of the canonical basis vectors of $\mathcal{H}_{\vec\kappa}(\mathbb{C}^d)$. There exists a constant $M=M(d,\vec\kappa,P_0)>0$ such
that for any $F\in \mathcal{H}_{\vec\kappa}(\mathbb{C}^d)$ we have
\[
\inf_{|\lambda|=1}\lVert F - \lambda P_0 \rVert_{\gamma_d}
\;\le\;
M\,
\bigl\|\, \lvert F  \rvert - |P_0| \,\bigr\|_{\gamma_d}.
\]
\end{theorem}
\begin{remark}
    Theorem~\ref{thm:local_Gen} has been autoformalized in Lean 4; we refer the reader to Sections~\ref{LLM Sect} and \ref{App} for a discussion of the Lean formalization and to \cite{angdinata2026abc,armstrong2026formalization,hariharan2026milestone,ho2026generalization,ilin2026semi} for recent examples of formalizations involving AI. 
\end{remark}
\begin{remark}
     We emphasize that the assumption that $P_0\in\text{span}\{\Phi_{\vec \kappa,\vec\alpha}\}_{\vec \alpha}$ is absolutely essential, as there is a dense set of $F_0\in\mathcal{H}_{\vec\kappa}(\mathbb{C}^d)$ for which 
    Theorem~\ref{thm:local_Gen} is false. We also remark that although each basis
element of $\text{span}\{\Phi_{\vec \kappa,\vec\alpha}\}_{\vec \alpha}$ factors as a tensor product, the proof of Theorem~\ref{thm:local_Gen} is not obtained by a formal tensorization of the one-dimensional statements. 
\end{remark}
As a direct consequence of Theorem \ref{main:local}, we are able to prove a \emph{sharp stability version} of the log-Sobolev inequality for entire functions. More specifically, if $F \in \mathcal{F}^2(\mathbb{C})$ with $\|F\|_{\gamma}= 1$, the logarithmic Sobolev inequality for entire functions states that 
\[
 \int_{\mathbb{C}} |F'(z)|^2 \, d\gamma(z)
  \;\ge\;
  \int_{\mathbb{C}} |F(z)|^2 \ln |F(z)|^2 \, d\gamma(z),
\]
where $d\gamma$ denotes the Gaussian measure. This estimate is directly tied to the classical logarithmic Sobolev inequality (see, e.g., \cite{Gross1975,Stam1959,Weissler1978} for the first works dealing with such an estimate), and its stability has been a recent subject of interest. Indeed, we first highlight \cite{DolbeaultEstebanFigalliFrankLoss2025}, where the authors prove asymptotically sharp stability estimates for the Sobolev inequality, which, by a limiting argument, yields stability estimates for the classical log-Sobolev inequality. More recently, R.~Frank, F.~Nicola and P.~Tilli \cite{FrankNicolaTilli2025} showed that the same estimate holds also in the aforementioned \emph{entire} case, obtaining the following result. 
\begin{theorem}[Frank--Nicola--Tilli, {\cite[Theorem~14]{FrankNicolaTilli2025}}]\label{thm:FNTLogSob}
There exists a constant $c_\ast > 0$ such that, for every
$F \in \mathcal{F}^2(\mathbb{C})$ with $\|F\|_{\gamma} = 1$, we have
\begin{equation*}
  \int_{\mathbb{C}} |F'(z)|^2 \, d\gamma(z)
  \;\ge\;
  \int_{\mathbb{C}} |F(z)|^2 \ln |F(z)|^2 \, d\gamma(z)
  \;+\;
  c_\ast \!\!\inf_{\substack{\beta \in \mathbb{C} \\ |c|=1}}\!
  \bigl\| F - c\, e^{\beta z - |\beta|^2/2} \bigr\|_{\gamma}^{2}.
\end{equation*}

\end{theorem}

As we shall see below, Theorem \ref{thm:FNTLogSob} follows almost directly from Theorem~\ref{main:local} and \cite[Corollary~1.2]{DolbeaultEstebanFigalliFrankLoss2025}.

\subsection{Explanation of LLM usage and Lean}\label{LLM Sect}
Here we explain the events that led to the results in this paper.
 \vspace{0.5em}

\textbf{LLM Usage.} The problem under consideration had previously been studied by several of the authors over the past years, with varying degrees of success. On 28 March 2026, the authors were experimenting with ChatGPT Plus (GPT 5.4) and, in a multi-turn conversation, became intrigued by a (false) reduction of the problem to a circle estimate (which one may view as a primitive version of the results in Section~\ref{subsec:circle}). The authors then tried to understand this reduction, further prompting GPT until they realized that a modification of the proposed argument could be salvaged and would likely lead to a proof. We then sketched the full proof of Theorem~\ref{thm:orthog} by ourselves, which naively lost a $\log$-factor in the stability estimate, and asked GPT 5.4 in Codex to refine it. This ultimately resulted in a complete proof of Theorem~\ref{thm:orthog}, which we knew implied Theorem~\ref{main:local}. We then autoformalized the polynomial version of this result in Lean 4 using Claude Opus 4.6. 
 \vspace{0.5em}

In hindsight, the fact that the structure of the problem is more easily visible in polar coordinates should have been evident. In fact, spiral sampling is commonly used in experimental applications. Moreover, the final proof is not particularly difficult, and once the structure was made clear, the authors could have completed the proof of  Theorem~\ref{main:local} without any additional assistance. What is interesting, however, is that the LLMs were able to incrementally generalize Theorem~\ref{main:local} to Theorem~\ref{thm:local_Gen} with very little human mathematical input. Of course, it should be emphasized that the authors strongly believed that the method would generalize. Moreover, as we will see below, significant challenges arose on the Lean side when verifying the full result.
 \vspace{0.5em}

As alluded to above, several of the intermediate results in this paper were known to the authors for a long time. For example, the ``local-local to local" reduction in Proposition~\ref{prop:local-reduction} is relatively standard. The orthogonal reduction in Theorem~\ref{thm:reduction}  is also just a natural variant of \cite[Theorem 3.10]{MR4800909} and was not produced by an LLM.
 \vspace{0.5em}

\textbf{Lean.} The main formal content of the results in this paper has been verified in Lean 4 \cite{lean4} using the Mathlib library \cite{mathlib2020}.  Here, we point out some interesting features of the formalization that may be useful for other problems. We refer the reader to Section~\ref{App} for a formalized statement of our main theorem.
 \vspace{0.5em}

In the beginning, we carried out the Lean verification only for finite linear combinations of the basis elements. Our justification for this was threefold. First, the density argument is completely standard -- see Section~\ref{subsec:local-reduction}. Second, the density argument is a general fact that is easy to formalize in a systematic way. Finally, we found that working directly with arbitrary polyanalytic functions rather than ``finite" polynomials posed unnecessary challenges in the Lean formalization.
 \vspace{0.5em}

The formalization process was completed with extensive assistance from LLMs; namely, the first result was formalized with Claude Opus 4.6 in Claude Code and later generalizations with GPT 5.4 and 5.5 through Codex, in all cases using Oliver Dressler's MCP \cite{DresslerMCP} under the supervision of the authors. The authors checked the statements of the main theorems and provided support whenever the formal development encountered blockers, such as statement mismatches between the blueprint and the corresponding results in the Lean files, or overly strong statements that made the proof substantially harder to fill in. We found, for example, that it was worthwhile to make all statements as quantitative as possible and to avoid  compactness arguments -- remnants of this insistence can be (intentionally) witnessed by the explicit choice of constants in the proof, which we did not try to optimize. For clarity, the reviewer facing statement of the main theorem in Section~\ref{App} was rewritten to assert the existence of a constant rather than to choose an explicit constant. However, some of the less general statements remain quantitative, namely the one for the Fock space case, which is an exact copy of Theorem \ref{thm:local}. 
 \vspace{0.5em}

The full Lean code is available on the GitHub repository, which can be accessed from here: https://github.com/susannabertolini/PhaseRetrieval. We remark that we put substantial effort into properly presenting the \emph{statement} of the main theorem in Lean (see Section~\ref{App}) as well as in decoupling the Lean statement of the theorem from the LLM generated proof (which in itself is of low quality). The main difficulties of the formalization arose from incorrect translations from the blueprint to the scaffolding: this led to incorrect or imprecise scaffolding files even with an extremely detailed blueprint. Another recurring issue was the tendency of the LLMs to incorporate unnecessary Hilbert-space structure whenever the blueprint mentioned it for context, even if the main theorem was clearly stated for finite linear combinations only. This introduced many density arguments that did not play a role in the proof of the main theorem but created blockers that had to be recognized as irrelevant and removed before being able to complete the formalization. 
\subsection{Acknowledgments}
J.P.G.R.~was supported by the FCT through project SHADE (project 2023.17881.ICDT, DOI  10.54499 / 2023.17881.ICDT) and by FAPERJ through the JCNE grant no.~SEI-260003/020475/2025.
 \vspace{0.5em}

The authors thank Alessio Figalli, Javier Gomez-Serrano and Joaquim Serra for their support and guidance during the preparation of this article. J.D.~and M.T.~also thank Manuel Guizar-Sicairos and Abraham Levitan for hosting them at PSI and for the instructive discussions on the experimental aspects of phase retrieval. Finally, we thank the Banach lattice and phase retrieval communities for their encouragement to pursue this line of research. 
 \vspace{0.5em}

The latter parts of this work were conducted while J.D.~and M.T.~were in residence at the Institute for Computational and Experimental Research in Mathematics in Providence, RI, during the \emph{Techniques and Tools for the Formalization of Analysis} program. This program was supported by the National Science Foundation under Grant No.~DMS-2424556. The authors are grateful to the organizers as well as to several participants for their helpful comments on the exposition of the formalized statement in Section \ref{App}.

\section{Proof of Theorem~\ref{main:local}}
Here, we present the proof of Theorem~\ref{main:local}. We begin by giving an outline of the proof. After this, we present some general reductions. Finally, we tackle the main estimate specific to Fock space phase retrieval. 

\subsection{Outline of the proof}\label{sec:strategy}
Below, we give an outline of the logical structure of the proof of Theorem \ref{main:local} as well as a brief overview of the key ideas. We will elaborate on each point in more detail in the corresponding subsection. 
 \vspace{0.5em}

\noindent\textbf{The local reduction
(Section~\ref{subsec:local-reduction}).}\;
The main goal of this subsection will be to show that Theorem \ref{main:local} is implied by Theorem \ref{thm:local}. This will follow from a more general argument that holds in any reproducing kernel Hilbert space (which the Fock space is a particular example of). Essentially, we will show that if one
can prove phase retrieval at a function \(F_0\) and establish local stability
with respect to very small perturbations of \(F_0\), then one can prove local
stability for general perturbations of \(F_0\). In the context of Theorem \ref{main:local}, one takes $F_0=1$, but we prefer to state this argument in a more abstract form so that it also applies (for instance) in the context of Theorem \ref{thm:local_Gen}. 
 \vspace{0.5em}

\noindent
\textbf{The orthogonal reduction (Section~\ref{subsec:reduction}).}\;
Another general argument will be used to show that if orthogonal perturbations
of a unit vector $F_0$ satisfy a certain coercivity estimate with constant~$M$,
then all sufficiently small perturbations (after optimizing over a global
phase) satisfy the same coercivity estimate with a comparable constant $\widetilde{M} = O(M)$. We leave the precise statement and proof of this reduction to Section \ref{subsec:reduction}, but we emphasize that it is just the natural local variant of \cite[Theorem 3.10]{MR4800909}. In the special case of the Fock space $\mathcal{F}^2(\mathbb{C})$, this observation will enable us to reduce the proof of Theorem
\ref{thm:local} to the following ``orthogonal
coercivity" estimate, whose proof is the technical heart of the paper.

\begin{theorem}[Fock-space coercivity]\label{thm:orthog}
There exists an absolute constant $C$ such that for every $D \ge 1$ and every polynomial
$U(z) = \sum_{n=1}^D a_n z^n$ with $U(0) = 0$, we have
\[
  \norm{U}_{\mathcal{F}} \le C\,\norm{\rho(U)}_{\mathcal{F}}.
\]
Here, $\rho(w)\coloneq ||1+w|-1|$ and one can take $C = 4600$.
\end{theorem}
\begin{remark}
  The condition $U(0) = 0$ means that $U$ is orthogonal to the constant one
function in $\mathcal{F}^2(\mathbb{C})$. Thus, the orthogonal reduction tells us that Theorem~\ref{thm:orthog} is strong enough to recover our main results. We  remark that the constants in Theorem~\ref{thm:local} are obtained by inserting $C=4600$ into Theorem~\ref{thm:reduction}.
\end{remark}

We now outline the steps needed to prove Theorem~\ref{thm:orthog}.
 \vspace{0.5em}

\noindent
\textbf{Polar coordinates.}\;
A simple but important idea in our proof below is to represent the polynomial $U$ in Theorem \ref{thm:orthog} in polar coordinates. This will enable us to rephrase many of our estimates in terms of estimates for trigonometric polynomials (where we have more direct access to Fourier analytic tools). We further note that the Gaussian integral decomposes as
\begin{equation}\label{eq:polar}
\frac{1}{\pi}\int_{\mathbb{C}} F(z)\, \mathrm{e}^{-\lvert z \rvert^2}\,\mathrm{d} m(z)
\;=\;
2\int_0^\infty r\,\mathrm{e}^{-r^2}
\biggl(\int_{S^1} F(r\mathrm{e}^{\mathrm{i} t})\,\mathrm{d}\sigma(t)\biggr)\mathrm{d} r,
\end{equation}
where $\mathrm{d}\sigma = \mathrm{d} t/(2\pi)$ is the probability Haar measure on the
circle. Moreover, both norms in Theorem~\ref{thm:orthog} decompose in this way, so it
suffices to compare the inner circle integrals for each radius~$r$. This simple observation will be very useful in the forthcoming analysis.
 \vspace{0.5em}

\noindent
\textbf{Circle estimates (Section~\ref{subsec:circle}).}\;
On a circle of radius~$r$, the function $t \mapsto U(r\mathrm{e}^{\mathrm{i} t})$ is a
trigonometric polynomial with strictly positive frequencies (since
$U(0) = 0$).  We prove two estimates relating the $L^2$ norm of such a
polynomial~$P$ to that of $\rho \circ P$:
\begin{enumerate}[(i)]
\item A \emph{local circle estimate} that works for any finite set of
  positive frequencies, at the cost of a constant depending on how many
  frequencies are present.

\item A \emph{high-frequency estimate} that applies when the lowest
  frequency is much larger than the bandwidth of frequencies, giving a universal constant
  independent of the number of frequencies.
\end{enumerate}
 \vspace{0.5em}

\noindent
\textbf{Frequency localization (Section~\ref{subsec:blocks}).}\; The polynomial $U$ may have up to $D$ frequencies, but the monomial~$z^n$
concentrates its Gaussian mass near the annulus $\lvert z \rvert \approx \sqrt{n}$.
We exploit this by partitioning the frequencies into blocks, each localized
to a specific annulus.  For each annulus, only a bounded window of nearby
blocks contributes significantly, and the energy leaking from distant blocks
decays as a Gaussian in the block distance.
 \vspace{0.5em}

\noindent
\textbf{Finishing the proof (Section~\ref{subsec:assembly}).}\;
Here, we assemble the estimates of the previous two subsections to complete the proof of the orthogonal coercivity bound in Theorem \ref{thm:orthog}. The observation is that on each annulus, the nearby blocks form a trigonometric polynomial to which
one of the two circle estimates applies.  This gives a uniform bound on the
$\rho$-norm of the ``local'' polynomial.  We then use the $1$-Lipschitz
property of $\rho$ to pass from the local polynomial to the full $U$,
integrate over all annuli, and absorb the exponentially small leakage.

\subsection{The local reduction}\label{subsec:local-reduction}
In this subsection, we  reduce Theorem~\ref{main:local} to Theorem~\ref{thm:local}. The reduction is abstract and only relies on the RKHS structure of  the Fock space. The first step is to show that local stable phase retrieval is implied by an
even more local variant. The second step is a simple density argument. 

\begin{proposition}[Local reduction]\label{prop:local-reduction}
Let \(\mathcal{H}\) be a reproducing kernel Hilbert space of real or complex-valued
functions isometrically embedded in \(L^2(\Omega,\mu)\) where $(\Omega,\mu)$ is some measure space. Fix \(F\in \mathcal{H}\) and
\(\epsilon>0\). Suppose that there exists a constant \(C>0\) such that for
all \(G\in \mathcal{H}\) satisfying
\[
\inf_{\lvert \lambda \rvert=1}\lVert F-\lambda G \rVert_{L^2(\Omega)}<\epsilon
\]
we have
\[
\inf_{\lvert \lambda \rvert=1}\lVert F-\lambda G \rVert_{L^2(\Omega)}
\le
C\lVert \lvert F \rvert-\lvert G \rvert \rVert_{L^2(\Omega)}.
\]
Suppose also that phase retrieval holds at \(F\), i.e., if \(G\in \mathcal{H}\) and
\(\lvert F \rvert=\lvert G \rvert\)~a.e., then there exists \(\lambda\in\mathbb{C}\) with
\(\lvert \lambda \rvert=1\) so that \(F=\lambda G\). Then there exists a constant
\(C'>0\) such that for all \(G\in \mathcal{H}\),
\[
\inf_{\lvert \lambda \rvert=1}\lVert F-\lambda G \rVert_{L^2(\Omega)}
\le
C'\lVert \lvert F \rvert-\lvert G \rvert \rVert_{L^2(\Omega)}.
\]
\end{proposition}

\begin{proof}
We need to verify that there exists a constant \(C'\) such that for all
\(G\in \mathcal{H}\),
\[
\inf_{\lvert \lambda \rvert=1}\lVert F-\lambda G \rVert_{L^2(\Omega)}
\le
C'\lVert \lvert F \rvert-\lvert G \rvert \rVert_{L^2(\Omega)}.
\]
First suppose that \(\lVert G \rVert_{L^2(\Omega)}\ge4\lVert F \rVert_{L^2(\Omega)}\).
Then
\[
\lVert \lvert F \rvert-\lvert G \rvert \rVert_{L^2(\Omega)}
\ge
\lVert G \rVert_{L^2(\Omega)}-\lVert F \rVert_{L^2(\Omega)}
\ge
\frac34\lVert G \rVert_{L^2(\Omega)},
\]
while
\[
\inf_{\lvert \lambda \rvert=1}\lVert F-\lambda G \rVert_{L^2(\Omega)}
\le
\lVert F \rVert_{L^2(\Omega)}+\lVert G \rVert_{L^2(\Omega)}
\le
\frac54\lVert G \rVert_{L^2(\Omega)}.
\]
Thus the desired inequality holds in this case with any \(C'\ge5/3\).
Therefore, we may restrict attention to \(G\) with
\(\lVert G \rVert_{L^2(\Omega)}<4\lVert F \rVert_{L^2(\Omega)}\).
 \vspace{0.5em}

Now suppose that the desired inequality is true for all \(G\) satisfying
\[
\inf_{\lvert \lambda \rvert=1}\lVert F-\lambda G \rVert_{L^2(\Omega)}<\epsilon,
\]
but is not true in general. Then for every sufficiently large \(n\) we can find \(G_n\in \mathcal{H}\)
with
\[
\inf_{\lvert \lambda \rvert=1}\lVert F-\lambda G_n \rVert_{L^2(\Omega)}\ge\epsilon,
\]
\[
\inf_{\lvert \lambda \rvert=1}\lVert F-\lambda G_n \rVert_{L^2(\Omega)}
>
n\lVert \lvert F \rvert-\lvert G_n \rvert \rVert_{L^2(\Omega)},
\]
and also \(\lVert G_n \rVert_{L^2(\Omega)}<4\lVert F \rVert_{L^2(\Omega)}\). This latter
inequality guarantees that
\(\inf_{\lvert \lambda \rvert=1}\lVert F-\lambda G_n \rVert_{L^2(\Omega)}\) is uniformly
bounded, and hence
\[
\lVert \lvert F \rvert-\lvert G_n \rvert \rVert_{L^2(\Omega)}\to0.
\]
In particular, after passing to a subsequence, we may assume that
\[
\lvert G_{n_k}(x) \rvert\to\lvert F(x) \rvert
\qquad\text{for almost every }x\in\Omega.
\]
Since \(\lVert G_{n_k} \rVert_{L^2(\Omega)}\) is bounded independently of \(k\) and
the embedding of \(\mathcal{H}\) into \(L^2(\Omega,\mu)\) is isometric, the sequence
\((G_{n_k})\) is bounded in \(\mathcal{H}\). Passing to a further subsequence if
necessary, we may assume that \(G_{n_{k_\ell}}\rightharpoonup G\) weakly in
\(\mathcal{H}\). Since point evaluations are bounded on \(\mathcal{H}\), we have
\[
G_{n_{k_\ell}}(x)\to G(x)
\]
for every \(x\in\Omega\). Hence \(\lvert G_{n_{k_\ell}}(x) \rvert\to\lvert G(x) \rvert\) for
every \(x\in\Omega\). This forces \(\lvert G \rvert=\lvert F \rvert\) almost everywhere. By the
phase-retrieval hypothesis, \(G=\lambda F\) for some unimodular scalar \(\lambda\).
 \vspace{0.5em}

Collecting what we know, \(G_{n_{k_\ell}}\) converges weakly in \(\mathcal{H}\) to
\(\lambda F\). Moreover, from
\(\lVert \lvert F \rvert-\lvert G_n \rvert \rVert_{L^2(\Omega)}\to0\), we see that
\[
\lVert G_n \rVert_{L^2(\Omega)}=\lVert \lvert G_n \rvert \rVert_{L^2(\Omega)}
\to
\lVert \lvert F \rvert \rVert_{L^2(\Omega)}
=\lVert \lambda F \rVert_{L^2(\Omega)}.
\]
Because the embedding is isometric, weak convergence in \(\mathcal{H}\) together with
convergence of the \(L^2\)-norms to the norm of the weak limit implies strong
convergence in \(L^2\). Thus \(G_{n_{k_\ell}}\to\lambda F\) in
\(L^2(\Omega)\). For all sufficiently large \(\ell\), this contradicts (after relabeling the unimodular scalar)
\[
\inf_{\lvert \lambda \rvert=1}
\lVert F-\lambda G_{n_{k_\ell}} \rVert_{L^2(\Omega)}
\ge \epsilon.
\]
\end{proof}
\subsubsection{Proof of Theorem \ref{main:local}} Using the above reduction, we immediately deduce Theorem \ref{main:local} from Theorem \ref{thm:local} by density and by applying Proposition \ref{prop:local-reduction} with $F=1$ and $\mathcal{H}=\mathcal{F}^2(\mathbb{C})$. Indeed, Theorem \ref{thm:local} immediately extends from polynomial perturbations $P$ to arbitrary elements $\widetilde{P}\in \mathcal{F}^2(\mathbb{C})$ with $\|\widetilde{P}\|_{\mathcal{F}} < \delta$ by density. In that density passage, we choose polynomials $P_n \to \widetilde{P}$, apply Theorem \ref{thm:local}, pass to a subsequence of phases $w_n \to w$, and use continuity of $t \mapsto |1+t|$ in $L^2$.
Phase retrieval at the constant function in the Fock space holds because $|G| = 1$ a.e.~forces the entire function $G$ to be a unimodular constant. Hence, it is easy to see that Proposition \ref{prop:local-reduction} is applicable.
\subsection{The orthogonal reduction}\label{subsec:reduction}
In this subsection, we work in the generality of a complex Hilbert space $\mathcal{H}$
isometrically embedded in $L^2(\mu)$ for some measure space $(\Omega, \mu)$.
We show that \emph{orthogonal coercivity} -- a stability estimate for
perturbations perpendicular to a fixed unit vector -- implies \emph{local
stability} for all sufficiently small perturbations, after optimizing over a
global phase.  The argument is based on function space theory and uses no
structure specific to the Fock space. 

\begin{theorem}[Orthogonal reduction]\label{thm:reduction}
\par
Let $\mathcal{H}$ be a complex Hilbert space with norm $\lVert \cdot \rVert$, isometrically
embedded in $L^2(\Omega,\mu)$.  Let $f_0 \in \mathcal{H}$ with $\lVert f_0 \rVert = 1$ and
suppose that there exists $M > 0$ such that
\begin{equation}\label{eq:orthog-hyp}
\lVert g \rVert
\;\le\;
M\, \bigl\|\, \lvert f_0 + g \rvert - \lvert f_0 \rvert \,\bigr\|_{L^2}
\qquad
\text{for all } g \perp f_0.
\end{equation}
Then there exist $\delta > 0$ and $\widetilde{M} > 0$, depending only
on~$M$, such that
for every $h \in \mathcal{H}$ with $\lVert h \rVert \le \delta$ and
$\langle h,f_0\rangle \in \mathbb{R}$, we have
\[
\lVert h \rVert
\;\le\;
\widetilde{M}\, \bigl\|\, \lvert f_0 + h \rvert - \lvert f_0 \rvert \,\bigr\|_{L^2}.
\]
One may take $\delta = 1/(M+1)$ and $\widetilde{M} = 5M + 3$.
\end{theorem}

\begin{remark}
The condition $\langle h,f_0\rangle \in \mathbb{R}$ is a phase normalization: it says that
$h$ has been rotated so that its component along $f_0$ is real.  In the
application to Theorem~\ref{thm:local}, this is achieved by choosing the
optimal phase $w$ so that $\langle w(f_0 + p) - f_0,f_0\rangle$ is real, and aligning the phase of the inner product
$\langle f_0+p,f_0\rangle$.
\end{remark}

\begin{proof}
Write $m = \bigl\|\, \lvert f_0 + h \rvert - \lvert f_0 \rvert \,\bigr\|_{L^2}$ for the right-hand side of our inequality. We break the proof into four simple steps.
 \vspace{0.5em}

\noindent
\emph{Step 1: Orthogonal decomposition.}\;
Set $a = \operatorname{Re}\langle h,f_0\rangle = \langle h,f_0\rangle$ and
$g = h - a f_0$.  Then $g \perp f_0$, $h = af_0 + g$ and
$\lVert h \rVert^2 = a^2 + \lVert g \rVert^2$, so
\begin{equation}\label{eq:triangle}
\lVert h \rVert \;\le\; \lvert a \rvert + \lVert g \rVert.
\end{equation}

\noindent
\emph{Step 2: Control of $a$.}\;
The pointwise identity
$\lvert f_0 + h \rvert^2 - \lvert f_0 \rvert^2
= 2\,\operatorname{Re}(h\,\overline{f_0}) + \lvert h \rvert^2$
integrates to
\[
\int_\Omega \bigl(\lvert f_0 + h \rvert^2 - \lvert f_0 \rvert^2\bigr)\,\mathrm{d}\mu
= 2a + \lVert h \rVert^2.
\]
On the other hand, the factorization
$\lvert f_0 + h \rvert^2 - \lvert f_0 \rvert^2
= \bigl(\lvert f_0 + h \rvert - \lvert f_0 \rvert\bigr)
  \bigl(\lvert f_0 + h \rvert + \lvert f_0 \rvert\bigr)$
and Cauchy--Schwarz give
\[
\lvert 2a + \lVert h \rVert^2 \rvert
\;\le\;
m \cdot \bigl\|\, \lvert f_0 + h \rvert + \lvert f_0 \rvert \,\bigr\|_{L^2}
\;\le\;
m\,(2 + \lVert h \rVert).
\]
Hence,
\begin{equation}\label{eq:a-bound}
\lvert a \rvert
\;\le\;
\frac{2 + \lVert h \rVert}{2}\, m
+ \frac{\lVert h \rVert^2}{2}.
\end{equation}

\noindent
\emph{Step 3: Control of $g$.}\;
The reverse triangle inequality
$\bigl\lvert \lvert f_0 + g \rvert - \lvert f_0 + h \rvert \bigr\rvert
\le \lvert g - h \rvert = \lvert a \rvert\,\lvert f_0 \rvert$
gives
$\bigl\|\, \lvert f_0 + g \rvert - \lvert f_0 \rvert \,\bigr\|_{L^2}
\le \lvert a \rvert + m$.
The hypothesis~\eqref{eq:orthog-hyp} then yields
\begin{equation}\label{eq:g-bound}
\lVert g \rVert \;\le\; M\,(\lvert a \rvert + m).
\end{equation}

\noindent
\emph{Step 4: Completing the proof.}\;
Substituting~\eqref{eq:a-bound} and~\eqref{eq:g-bound}
into~\eqref{eq:triangle}, we see that
\[
\lVert h \rVert
\;\le\;
\lvert a \rvert + M(\lvert a \rvert + m)
= (M+1)\lvert a \rvert + Mm
\;\le\;
\underbrace{\Bigl(M + \frac{(M+1)(2+\lVert h \rVert)}{2}\Bigr)}_{=:\, C(\lVert h \rVert)}
m
\;+\; \frac{M+1}{2}\,\lVert h \rVert^2.
\]
Since $\lVert h \rVert \le \delta$, we have $\lVert h \rVert^2 \le \delta\,\lVert h \rVert$, so
\[
\lVert h \rVert
\;\le\; C(\delta)\, m + \frac{(M+1)\delta}{2}\,\lVert h \rVert.
\]
Choosing $\delta = 1/(M+1)$ makes $(M+1)\delta/2 = 1/2$, giving
$\lVert h \rVert/2 \le C(\delta)\, m$ and therefore
\[
\lVert h \rVert \;\le\; 2\, C(\delta)\, m.
\]
A short computation shows that
$C(\delta) = M + (M+1)(2 + \delta)/2 < (5M+3)/2$, so one may take
$\widetilde{M} = 5M + 3$.
\end{proof}

\begin{corollary}[Fock-space application]\label{cor:fock-local}
Taking $f_0 = 1$ (the constant function), $\mathcal{H} = \mathcal{F}^2(\mathbb{C})$, and
$M = 4600$ (from Theorem~\ref{thm:orthog}), the orthogonal reduction gives
$\delta = 1/4601$ and $\widetilde{M} = 23003$.
This is Theorem~\ref{thm:local}.
\end{corollary}

\begin{proof}
The condition $g \perp 1$ in $\mathcal{F}^2(\mathbb{C})$ means precisely that
$g(0) = 0$, i.e., $g$ has no constant term.
Theorem~\ref{thm:orthog} provides the orthogonal coercivity
hypothesis~\eqref{eq:orthog-hyp} with $M = 4600$.
The optimal phase $w$ is chosen so that
$h = w(1 + P) - 1$ satisfies $\langle h,1\rangle \in \mathbb{R}$ and
$\lVert h \rVert_{\mathcal{F}} \le \delta$; the existence of such a $w$ is straightforward.
Theorem~\ref{thm:reduction} then yields the conclusion of
Theorem~\ref{thm:local}. 
\end{proof}
Our goal for the remaining subsections will  be to establish the key bound in Theorem  \ref{thm:orthog}.
 \subsection{Circle estimates}\label{subsec:circle}

In this section we prove two estimates controlling the $L^2$ norm of a
trigonometric polynomial $P$ on the circle in terms of the $L^2$ norm of
$\rho \circ P$.  Both exploit the fact that all frequencies of~$P$ are
strictly positive (which is a consequence of our reduction to the case $U(0)=0$) but they apply in
complementary regimes.  The \emph{local circle estimate}
(Proposition~\ref{prop:local-circle}) works for any frequency set and gives a
constant proportional to $\sqrt{L}$, where $L$ is the number of frequencies.
The \emph{high-frequency estimate} (Proposition~\ref{prop:high-freq}) applies
when the lowest frequency is much larger than the bandwidth of frequencies, and gives an
absolute constant. Together they will provide a uniform bound on every
annulus.
 \vspace{0.5em}

Throughout this section, $P$ denotes a trigonometric polynomial
$P(t) = \sum_{n \in E} b_n \mathrm{e}^{\mathrm{i} nt}$ for a finite set $E \subset \mathbb{Z}_{>0}$ 
and all norms are taken in $L^2(S^1, \mathrm{d}\sigma)$. Recall as well that the function $\rho(w) = \bigl\lvert \lvert 1+w \rvert - 1 \bigr\rvert$
 is $1$-Lipschitz, and in particular we have $\rho(w) \le \lvert w \rvert$ and
\begin{equation}\label{eq:rho-lip}
\rho(w) \;\ge\; \rho(z) - \lvert w - z \rvert.
\end{equation}

\subsubsection{The local circle estimate}\label{sec:local}

The following elementary inequality will be  used several times in the forthcoming calculations, so we record it here for convenience.

\begin{lemma}\label{lem:safe-square}
Let $a, b, c \ge 0$ with $a \ge b - c$.  Then
$a^2 \ge \tfrac{1}{2}\, b^2 - c^2$.
\end{lemma}

\begin{proof}
If $b \le c$ the right-hand side is nonpositive.  If $b > c$ then
$a \ge b-c > 0$ gives $a^2 \ge (b-c)^2 \ge \tfrac{1}{2}b^2 - c^2$, where
the last step is the identity
$(b-c)^2 - \tfrac{1}{2}b^2 + c^2 = \tfrac{1}{2}(b-2c)^2 \ge 0$.
\end{proof}

\begin{proposition}[Local circle estimate]\label{prop:local-circle} Let $E \subset \mathbb{Z}_{>0}$ with $\lvert E \rvert = L \ge 1$, let $b_n \in \mathbb{C}$ for
$n \in E$, and set $P(t) = \sum_{n \in E} b_n\,\mathrm{e}^{\mathrm{i} nt}$.  Then
\[
\lVert P \rVert_{L^2}^2
\;\le\;
144\, L \cdot \lVert \rho \circ P \rVert_{L^2}^2.
\]
\end{proposition}

\begin{proof}
Write $\alpha = \lVert P \rVert_{L^2}$ and $\beta = \lVert \rho \circ P \rVert_{L^2}$.
We show that $\beta \ge \alpha/(12\sqrt{L})$, which gives the result after
squaring. To prove this estimate, we will consider two different regimes; namely, the one where $\alpha$ is sufficiently small and its complement. The threshold between the two regimes will be taken to be
$\alpha=\alpha_0 = 1/(4\sqrt{L})$.
 \vspace{0.5em}

Two structural properties of $P$ underpin both cases.  First,
by Cauchy--Schwarz and Parseval we have the pointwise bound
$\lvert P(t) \rvert \le \sqrt{L}\, \alpha$ for all~$t$.
Second, since every $n \in E$ satisfies $n \ge 1$, the zeroth Fourier
coefficient of $P$ vanishes, i.e.,
$\int_{S^1} P\,\mathrm{d}\sigma = 0$.
More importantly, we have $n + m \ge 2$ for all $n, m \in E$, so
$\int_{S^1} P^2\,\mathrm{d}\sigma = 0$ as well.  Writing $P = u + \mathrm{i} v$ with
$u, v$ real-valued, this identity splits as
\begin{equation}\label{eq:equal-mass}
\int_{S^1} u^2\,\mathrm{d}\sigma
= \int_{S^1} v^2\,\mathrm{d}\sigma
= \frac{\alpha^2}{2}
\qquad\text{and}\qquad
\int_{S^1} uv\,\mathrm{d}\sigma = 0.
\end{equation}
The equipartition of mass between the real and imaginary parts of $P$ given by \eqref{eq:equal-mass} will be crucially used in the estimates below.
 \vspace{0.5em}

\noindent
\textbf{Small amplitudes} ($\alpha \le \alpha_0$).
Here, we have $\lVert P \rVert_\infty \le \sqrt{L}\,\alpha \le 1/4$.  For $\lvert w \rvert \le 1/4$
we also have $\lvert 1+w \rvert \ge 3/4$ and $\operatorname{Re}(w) \ge -1/4$, and therefore
$\lvert 1+w \rvert + 1 + \operatorname{Re}(w) \ge 3/2 > 0$.  The algebraic identity
\[
\lvert 1+w \rvert - 1 - \operatorname{Re}(w)
= \frac{\operatorname{Im}(w)^2}{\lvert 1+w \rvert + 1 + \operatorname{Re}(w)}
\]
shows that $\lvert 1+w \rvert - 1 = \operatorname{Re}(w) + \varepsilon$ with
$0 \le \varepsilon \le \operatorname{Im}(w)^2/(3/2) \le \lvert w \rvert^2$.  In particular,
$\rho(w) = \bigl\lvert \operatorname{Re}(w) + \varepsilon\bigr\rvert \ge
\lvert \operatorname{Re}(w) \rvert - \varepsilon$, so by 
Lemma~\ref{lem:safe-square} we have
\[
\rho(P)^2 \;\ge\; \tfrac{1}{2}\, u^2 - \lvert P \rvert^4.
\]
In other words, when $P$ is small in amplitude, most of the mass of $\rho(P)$ comes from the real part of $P$. One might a priori worry about the hypothetical possibility that if $P(t)$ were allowed to be purely imaginary for most $t$, the nonlinear function $\rho(P(t))$ might be small relative to $P(t)$. However, integrating the above and using property~\eqref{eq:equal-mass}, we see that we indeed have the coercive bound
$\beta^2 \ge \alpha^2/4 - \lVert P \rVert_\infty^2\, \alpha^2 \ge \alpha^2/4 - L\,\alpha^4$.
Since $L\alpha^2 \le 1/16$, we get $\beta^2 \ge 3\alpha^2/16$, so
$\beta \ge \sqrt{3}\,\alpha/4 \ge \alpha/(12\sqrt{L})$.
 \vspace{0.5em}

\noindent
\textbf{Large amplitudes} ($\alpha > \alpha_0$).
The fact that $P$ has mean zero (since it is supported on positive frequencies) also gives
$\int_{S^1}\lvert 1+P \rvert^2\,\mathrm{d}\sigma = 1 + \alpha^2$, so
\[
\alpha^2
= \int_{S^1}\bigl(\lvert 1+P \rvert^2 - 1\bigr)\,\mathrm{d}\sigma
\le \int_{S^1} \rho(P)\,\bigl(\lvert 1+P \rvert + 1\bigr)\,\mathrm{d}\sigma,
\]
where we used the identity $\bigl\lvert \lvert 1+P \rvert^2 - 1\bigr\rvert
= \bigl(\lvert 1+P \rvert + 1\bigr)\rho(P)$.  By Cauchy--Schwarz,
\[
\alpha^2 \;\le\; \beta \cdot \bigl(\sqrt{1+\alpha^2} + 1\bigr).
\]
If $\alpha \le 1$, then $\beta \ge \alpha^2/3$ and $\alpha > \alpha_0 = 1/(4\sqrt{L})$ gives
$\beta \ge \alpha/(12\sqrt{L})$.
If $\alpha > 1$, then $\beta \ge \alpha/3 \ge \alpha/(12\sqrt{L})$.
In both regimes, we have $\beta \ge \alpha/(12\sqrt{L})$.
\end{proof}

\subsubsection{Averaging over circles}\label{sec:rot-avg}
The local circle estimate grows with the number of frequencies~$L$.  To
eliminate this dependency for high-frequency polynomials, we will need 
to use the fact that the $\rho$-image of the circle sees a
quantitative fraction of its radius. More precisely, we have the following proposition.

\begin{proposition}\label{prop:rot-avg}
For all $r \ge 0$, we have
\[
\int_{S^1} \rho(r\mathrm{e}^{\mathrm{i}\theta})^2\,\mathrm{d}\sigma(\theta)
\;\ge\; \frac{r^2}{8}.
\]
\end{proposition}

\begin{proof}
We claim that on the arc $\lvert \theta \rvert \le \pi/4$ we have $\rho(r\mathrm{e}^{\mathrm{i}\theta}) \ge r/\sqrt{2}$. Indeed,
$\cos\theta \ge 1/\sqrt{2}$ gives
\[
\lvert 1 + r\mathrm{e}^{\mathrm{i}\theta} \rvert - 1
= \frac{r^2 + 2r\cos\theta}{\lvert 1+r\mathrm{e}^{\mathrm{i}\theta} \rvert + 1}
\ge \frac{r(r + \sqrt{2})}{r + 2}
\ge \frac{r}{\sqrt{2}}.
\]
 Since the numerator
is non-negative, we have $\lvert 1 + r\mathrm{e}^{\mathrm{i}\theta} \rvert \geq 1$ and hence $\rho$ equals
$\lvert 1+r\mathrm{e}^{\mathrm{i}\theta} \rvert - 1$ without the absolute value.  Restricting the
integral to this arc, we have
$\int \rho^2\,\mathrm{d}\sigma \ge \frac{1}{4}(r/\sqrt{2})^2 = r^2/8$.
\end{proof}

\subsubsection{The high-frequency estimate}\label{sec:high-freq}

When the polynomial oscillates rapidly, the Lipschitz property
of~$\rho$ and 
Proposition~\ref{prop:rot-avg} combine to give a much stronger bound.  The idea
is to factor $P(t) = \mathrm{e}^{\mathrm{i} Nt}\, Q(t)$ where $Q$ is a slowly varying
envelope. On short intervals, the phase $\mathrm{e}^{\mathrm{i} Nt}$ sweeps through
a full rotation while $Q$ barely changes. Proposition~\ref{prop:rot-avg} then
ensures that $\rho$ captures a positive fraction of $\lvert Q \rvert^2$.  We make this intuition precise in the estimate below.

\begin{proposition}[High-frequency estimate]\label{prop:high-freq} Let $N \ge 1$ and $L \ge 1$ satisfy $1343\, L^2 \le N^2$.
For any $b_0, \dots, b_{L-1} \in \mathbb{C}$, the polynomial
$P(t) = \sum_{m=0}^{L-1} b_m\, \mathrm{e}^{\mathrm{i}(N+m)t}$ satisfies
\[
\lVert P \rVert_{L^2}^2
\;\le\;
32\, \lVert \rho \circ P \rVert_{L^2}^2.
\]
\end{proposition}

\begin{proof}
Write $P(t) = \mathrm{e}^{\mathrm{i} Nt}\,Q(t)$ with
$Q(t) = \sum_{m=0}^{L-1} b_m\,\mathrm{e}^{\mathrm{i} mt}$, so $\lVert P \rVert = \lVert Q \rVert$.
Partition the circle into $N$ equal arcs
$J_k = [2\pi k/N,\; 2\pi(k+1)/N)$, and let
$q_k = N\!\int_{J_k}\! Q\,\mathrm{d}\sigma$ be the average of $Q$ on~$J_k$.
 \vspace{0.5em}

On each interval $J_k$ of length $h = 2\pi/N$, the Poincar\'e
inequality gives
$\int_{J_k}\lvert Q - q_k \rvert^2\,\mathrm{d}\sigma \le h^2\int_{J_k}\lvert Q' \rvert^2\,
\mathrm{d}\sigma$.\footnote{We use the inequality
$\int_0^h \lvert f - \bar f \rvert^2\,\mathrm{d} x
\le h^2\int_0^h\lvert f' \rvert^2\,\mathrm{d} x$
with the non-sharp constant~$1$; here $\bar f$ denotes the mean of~$f$
on~$[0,h]$.  The sharp 
constant is $1/\pi^2$. After the conversion of the Lean code to natural language, we saw that the LLMs emphasized that this better estimate would have improved the threshold from $1343$ to~$136$ but they did not use it as we insisted that the proof be as elementary as possible.}
Summing over $k$ and using Bernstein's inequality $\lVert Q' \rVert^2 \le L^2\lVert Q \rVert^2$
(since every frequency of~$Q$ has modulus at most~$L{-}1 < L$) we see that
\begin{equation}\label{eq:delta}
\delta \;:=\; \sum_{k} \int_{J_k}\lvert Q - q_k \rvert^2\,\mathrm{d}\sigma
\;\le\; \frac{4\pi^2 L^2}{N^2}\,\lVert Q \rVert^2.
\end{equation}
Note that
\begin{equation}\label{eq:parseval-partition}
\lVert Q \rVert^2
= \frac{1}{N}\sum_{k}\lvert q_k \rvert^2 + \delta.
\end{equation}
As $t$ traverses $J_k$, the argument $Nt$ traverses one full period.
Substituting $s = Nt$ and applying Proposition~\ref{prop:rot-avg} with
$r = \lvert q_k \rvert$ we obtain
\[
\int_{J_k}\rho\bigl(\mathrm{e}^{\mathrm{i} Nt} q_k\bigr)^2\,\mathrm{d}\sigma
= \frac{1}{N}\int_{S^1}\rho\bigl(\lvert q_k \rvert\mathrm{e}^{\mathrm{i} s}\bigr)^2\,\mathrm{d}\sigma(s)
\;\ge\; \frac{\lvert q_k \rvert^2}{8N}.
\]
On $J_k$, the Lipschitz bound~\eqref{eq:rho-lip} gives
$\rho(P) = \rho(\mathrm{e}^{\mathrm{i} Nt}Q)
\ge \rho(\mathrm{e}^{\mathrm{i} Nt}q_k) - \lvert Q - q_k \rvert$.
By Lemma~\ref{lem:safe-square},
$\rho(P)^2 \ge \tfrac{1}{2}\rho(\mathrm{e}^{\mathrm{i} Nt}q_k)^2 - \lvert Q - q_k \rvert^2$.
Integrating over $J_k$ and using the above we obtain
\[
\int_{J_k}\rho(P)^2\,\mathrm{d}\sigma
\;\ge\; \frac{\lvert q_k \rvert^2}{16N} - \int_{J_k}\lvert Q-q_k \rvert^2\,\mathrm{d}\sigma.
\]
Summing over $k$ and applying~\eqref{eq:parseval-partition}, it follows that
\[
\lVert \rho \circ P \rVert^2
\;\ge\; \frac{1}{16}\Bigl(\lVert Q \rVert^2 - \delta\Bigr) - \delta
= \frac{\lVert Q \rVert^2}{16} - \frac{17\delta}{16}.
\]
Substituting the bound~\eqref{eq:delta}, we arrive at
\[
\lVert \rho \circ P \rVert^2
\;\ge\;
\frac{1}{16}\Bigl(1 - \frac{68\pi^2 L^2}{N^2}\Bigr)\lVert Q \rVert^2.
\]
The condition $1343\,L^2 \le N^2$ gives $68\pi^2 L^2/N^2 \le 1/2$
(since $136\pi^2 < 1343$), and therefore
$\lVert \rho \circ P \rVert^2 \ge \lVert P \rVert^2/32$.
\end{proof}

\subsection{Frequency localization}\label{subsec:blocks}
The circle estimates from Section~\ref{subsec:circle} require the trigonometric
polynomial to have a controlled number of frequencies, yet $U$ may have
arbitrarily many.  The Gaussian weight is what resolves this issue, as the radial factor
$r \mapsto r^{2n+1}\mathrm{e}^{-r^2}$ concentrates sharply near $r = \sqrt{n + 1/2}$. This will allow us to group frequencies of $U$ into ``blocks''
approximately supported on disjoint annuli.
 \vspace{0.5em}

We begin with some general notation. For $\ell \ge 1$, we define 
\[
I_\ell = \bigl\{n \in \mathbb{N} : \ell^2 \le n < (\ell+1)^2\bigr\},
\]
which has $\lvert I_\ell \rvert = 2\ell+1$ elements. We then define 
$U_\ell(z) = \sum_{n \in I_\ell,\, n \le D} a_n z^n$. The blocks $I_1, I_2, \dots$ are pairwise disjoint and cover
$\{1, \dots, D\}$ up to index $\Lambda = \lfloor\sqrt{D}\rfloor$.  Clearly, we have
\begin{equation}\label{eq:fock-decomp}
\lVert U \rVert_{\mathcal{F}}^2 = \sum_{\ell=1}^{\Lambda} \lVert U_\ell \rVert_{\mathcal{F}}^2.
\end{equation}
Similarly, on each circle of radius~$r$, blocks with distinct frequency
supports are orthogonal, so we have
$\int_{S^1} U_{\ell_1}(r\mathrm{e}^{\mathrm{i} t})\,
\overline{U_{\ell_2}(r\mathrm{e}^{\mathrm{i} t})}\,\mathrm{d}\sigma(t) = 0$
for $\ell_1 \ne \ell_2$. Given $n\in I_\ell$ with $\ell \geq 1$, we define the radius $r_n = \sqrt{n+1/2}\in (\ell, \ell+1)$.  \vspace{0.5em}

Fix  $M \ge 1$.  For each index $j \ge 0$ we define 
\[
V_j = \sum_{\substack{1 \le \ell \le \Lambda \\ \lvert \ell - j \rvert \le M}} U_\ell
\qquad\text{and} \qquad
R_j = U - V_j.
\]
Note that $V_j$ gathers together the blocks whose peaks fall within $M$ steps of
annulus~$j$. For $j \ge M + 1$, the frequency support of $V_j$ lies in
$[(j{-}M)^2,\; (j{+}M{+}1)^2)$, and the number of frequencies is
$L_j = (2M{+}1)(2j{+}1)$.
 \vspace{0.5em}

The mass that a block $U_\ell$ contributes to a distant annulus decays
like $\mathrm{e}^{-(\lvert j-\ell \rvert-1)^2}$. Indeed, the function
$\varphi_n(r) = (2n{+}1)\log r - r^2$ (defined so that
$r^{2n+1}\mathrm{e}^{-r^2} = \mathrm{e}^{\varphi_n(r)}$) has a unique maximum at
$r_n$ and satisfies $\varphi_n''(r) \le -2$, which gives the
 bound
\begin{equation}\label{eq:phi-quad}
\mathrm{e}^{\varphi_n(r)}
\;\le\;
\mathrm{e}^{\varphi_n(r_n)}\, \mathrm{e}^{-(r - r_n)^2}.
\end{equation}
For $n \ge 1$, it is easy to see by using Stirling’s approximation that
$\mathrm{e}^{\varphi_n(r_n)} \le \mathrm{e}^{1/4}\, n!/2$.
Combining this with \eqref{eq:phi-quad},  we see that for $r \in
[j, j{+}1]$ and $n \in I_\ell$, we have
\begin{equation}\label{eq:monomial-bound}
r^{2n+1}\mathrm{e}^{-r^2}
\;\le\;
\frac{\mathrm{e}^{1/4}\, n!}{2}\,
\mathrm{e}^{-\operatorname{dist}(r_n,\,[j,j{+}1])^2},
\end{equation}
where $\operatorname{dist}(r_n, [j,j{+}1]) = \max(j - r_n,\, r_n - (j{+}1),\, 0)$.
By definition, $r_n \in (\ell, \ell{+}1)$, so this distance is at
least $(\lvert j - \ell \rvert - 1)_+$. Our next objective is to quantify the ``leakage" between the blocks.

\begin{proposition}\label{prop:single-leak}

For $\ell \ge 1$ and $j \ge 0$, we have
\[
\frac{1}{\pi}\int_{\{j \le \lvert z \rvert < j+1\}}
\lvert U_\ell(z) \rvert^2\, \mathrm{e}^{-\lvert z \rvert^2}\,\mathrm{d} m(z)
\;\le\;
\mathrm{e}^{1/4}\,
\mathrm{e}^{-(\lvert j-\ell \rvert-1)_+^2}\,
\lVert U_\ell \rVert_{\mathcal{F}}^2.
\]
\end{proposition}

\begin{proof}
Passing to polar coordinates and using the orthogonality of distinct
monomials on each circle, the left-hand side equals
$2\sum_{n \in I_\ell}\lvert a_n \rvert^2 \int_j^{j+1} r^{2n+1}\mathrm{e}^{-r^2}\mathrm{d} r$.
 Using~\eqref{eq:monomial-bound} to bound each radial integral  and summing gives the result.
\end{proof}

By summing over all blocks, we arrive at the following corollary.
\begin{corollary}\label{cor:total-leak} For $M \ge 2$, define $\eta_M = 2\mathrm{e}^{1/4}\sum_{m \ge M}\mathrm{e}^{-m^2}$.  Then
\[
\sum_{j \ge 0}\;
\frac{1}{\pi}\int_{\{j \le \lvert z \rvert < j+1\}}
\lvert R_j(z) \rvert^2\, \mathrm{e}^{-\lvert z \rvert^2}\,\mathrm{d} m(z)
\;\le\;
\eta_M \cdot \lVert U \rVert_{\mathcal{F}}^2.
\]
\end{corollary}

\begin{proof}
By the orthogonality on each annulus,
$\int \lvert R_j \rvert^2 = \sum_{\lvert \ell-j \rvert>M}\int\lvert U_\ell \rvert^2$.
Now apply Proposition~\ref{prop:single-leak}, swap the order of summation, and note that for $\lvert j-\ell \rvert > M \ge 2$
the decay factor is $\mathrm{e}^{-(\lvert j-\ell \rvert-1)^2}$.  The inner sum over
$j$ contributes at most $2\sum_{m \ge M}\mathrm{e}^{-m^2}$ and multiplying by
$\mathrm{e}^{1/4}$ and using~\eqref{eq:fock-decomp} gives the result.
\end{proof}


We can now state the uniform estimate that holds on every annulus.  The
function $\zeta \mapsto V_j(r\zeta)$ is a trigonometric polynomial on $S^1$
with all frequencies in $\mathbb{Z}_{>0}$, and we can apply either
Proposition~\ref{prop:local-circle} or Proposition~\ref{prop:high-freq} depending
on~$j$.

\begin{proposition}\label{prop:annulus} With $M = 5$, we have for every $j \ge 0$ and every $r \in [j, j{+}1]$,
\[
\int_{S^1}\lvert V_j(r\zeta) \rvert^2\,\mathrm{d}\sigma(\zeta)
\;\le\;
1620^2 \cdot
\int_{S^1}\rho\bigl(V_j(r\zeta)\bigr)^2\,\mathrm{d}\sigma(\zeta).
\]
\end{proposition}

\begin{proof}
When $r = 0$ both sides vanish (all monomials have positive degree).
For $r > 0$, the function $\zeta \mapsto V_j(r\zeta)$ is a trigonometric
polynomial with positive frequencies and coefficients $a_n r^n$. We keep the peculiar case distinction made during the autoformalization.
 \vspace{0.7em}

\noindent
\emph{Case $j \le 817$.}\;
With $M = 5$, the number of frequencies satisfies $L_j \le 11(2 \cdot 817 + 1)
< 18{,}000$.  By Proposition~\ref{prop:local-circle},
$\lVert V_j(r\,\cdot\,) \rVert^2 \le 144\, L_j \cdot
\lVert \rho \circ V_j(r\,\cdot\,) \rVert^2$, and
$144 \cdot 18{,}000 < 1620^2$.
 \vspace{0.7em}

\noindent
\emph{Case $j \ge 818$.}\;
The lowest frequency is $N_j = (j{-}5)^2$ and we have
$L_j = 11(2j{+}1)$.  The ratio $L_j^2/N_j^2$ is decreasing in~$j$
(for $j \ge 6$), so it suffices to verify the condition
$1343\,L_j^2 \le N_j^2$ of Proposition~\ref{prop:high-freq} at $j = 818$.
A direct computation confirms that $1343 \cdot (11 \cdot 1637)^2 < 813^4$.
The high-frequency estimate then gives
$\lVert V_j(r\,\cdot\,) \rVert^2 \le 32\,\lVert \rho \circ V_j(r\,\cdot\,) \rVert^2
\le 1620^2\,\lVert \rho \circ V_j(r\,\cdot\,) \rVert^2$.
\end{proof} 
\subsection{Proof of the orthogonal coercivity estimate}\label{subsec:assembly} With the estimates from Sections~\ref{subsec:circle} and \ref{subsec:blocks} in hand, the proof of
Theorem~\ref{thm:orthog} is short.

\begin{proof}[Proof of Theorem~\ref{thm:orthog}]
Fix $M = 5$ throughout, and write $B = 1620$ for the constant from
Proposition~\ref{prop:annulus}.  Let $V_j$ and $R_j$ be as above.
 \vspace{0.5em}

\noindent
For $r \in [j, j{+}1]$, Proposition~\ref{prop:annulus} gives
\begin{equation}\label{eq:annulus-applied}
\int_{S^1}\lvert V_j(r\zeta) \rvert^2\,\mathrm{d}\sigma
\;\le\;
B^2\int_{S^1}\rho\bigl(V_j(r\zeta)\bigr)^2\,\mathrm{d}\sigma.
\end{equation}
Since $\rho$ is $1$-Lipschitz and $U = V_j + R_j$, we have
\[
\rho\bigl(V_j(z)\bigr)
\le \rho\bigl(U(z)\bigr) + \lvert R_j(z) \rvert.
\]
Using the inequality $(a+b)^2 \le 2a^2 + 2b^2$ and substituting
into~\eqref{eq:annulus-applied}, we see that
\begin{equation}\label{eq:Vj-to-U}
\int_{S^1}\lvert V_j(r\zeta) \rvert^2\,\mathrm{d}\sigma
\;\le\;
2B^2\!\int_{S^1}\rho(U(r\zeta))^2\,\mathrm{d}\sigma
\;+\; 2B^2\!\int_{S^1}\lvert R_j(r\zeta) \rvert^2\,\mathrm{d}\sigma.
\end{equation}
Applying $\lvert U \rvert^2 \le 2\lvert V_j \rvert^2 + 2\lvert R_j \rvert^2$
and~\eqref{eq:Vj-to-U}, we arrive at
\begin{equation}\label{eq:circle-ineq}
\int_{S^1}\lvert U(r\zeta) \rvert^2\,\mathrm{d}\sigma
\;\le\;
4B^2\!\int_{S^1}\rho(U)^2\,\mathrm{d}\sigma
\;+\; (4B^2 + 2)\!\int_{S^1}\lvert R_j \rvert^2\,\mathrm{d}\sigma.
\end{equation}

\noindent
We now multiply~\eqref{eq:circle-ineq} by $2r\mathrm{e}^{-r^2}$, integrate $r$ over
$[j, j{+}1]$, and use the polar decomposition~\eqref{eq:polar}.  The
$\rho(U)$ term on the right-hand side does not depend on~$j$, while the $R_j$ term does.  Summing over $j \ge 0$, we conclude that
\[
\lVert U \rVert_{\mathcal{F}}^2
\;\le\;
4B^2\,\lVert \rho(U) \rVert_{\mathcal{F}}^2
\;+\;
(4B^2 + 2)\sum_{j \ge 0}\;
\frac{1}{\pi}\int_{\{j \le \lvert z \rvert < j+1\}}\!\!
\lvert R_j \rvert^2\mathrm{e}^{-\lvert z \rvert^2}\mathrm{d} m.
\]
By Corollary~\ref{cor:total-leak} with $M = 5$, the above sum is at
most $\eta_5\,\lVert U \rVert_{\mathcal{F}}^2$.  Hence
\begin{equation}\label{eq:absorb}
\lVert U \rVert_{\mathcal{F}}^2
\;\le\;
4B^2\,\lVert \rho(U) \rVert_{\mathcal{F}}^2
\;+\;
(4B^2 + 2)\,\eta_5\,\lVert U \rVert_{\mathcal{F}}^2.
\end{equation}
We now absorb the additional term in the above inequality.
Since $4B^2 + 2 < 10^7$ and $\eta_5 < 4 \times 10^{-11}$, the coefficient
of $\lVert U \rVert_{\mathcal{F}}^2$ on the right-hand side of~\eqref{eq:absorb} satisfies
$(4B^2 + 2)\,\eta_5 < 4 \times 10^{-4} < \tfrac{1}{2}$.
Moving this term to the left-hand side and using $1 - (4B^2+2)\eta_5 > 1/2$, we conclude that
\[
\lVert U \rVert_{\mathcal{F}}^2
\;\le\; 8B^2\,\lVert \rho(U) \rVert_{\mathcal{F}}^2
\;=\; 8 \cdot 1620^2\,\lVert \rho(U) \rVert_{\mathcal{F}}^2
\;\le\; 4600^2\,\lVert \rho(U) \rVert_{\mathcal{F}}^2.\qedhere
\]
\end{proof}

\section{Comments, generalizations and extensions}
\subsection{Generalizations of Theorem \ref{main:local}} Here, we briefly discuss some of the ideas that allow one to generalize Theorem~\ref{main:local}. To motivate the structure of this subsection, it is instructive to explain the story of how we came to Theorem~\ref{main:local} and later Theorem~\ref{thm:local_Gen}.
 \vspace{0.5em}

The collaboration on this project began when the fourth author privately shared an unpublished proof of a weaker version of Theorem~\ref{main:local}; namely,   with $F$ replaced by $F^2$ in every instance. The proof of this weaker result is instructive and will be published elsewhere. However, despite extensive efforts, we were unable to extend the method to prove Theorem~\ref{main:local}. This necessitated the new method that we presented above.
\vspace{0.5em}

After establishing Theorem~\ref{main:local}, we asked the LLMs to generalize the method further, which eventually resulted in Theorem~\ref{thm:local_Gen}. Importantly, however, we did not immediately move from Theorem~\ref{main:local} to Theorem~\ref{thm:local_Gen}, but instead proceeded incrementally. 
\vspace{0.5em}

The first step of the generalization was to prove the analogue of Theorem~\ref{main:local} for the first Hermite window at the first canonical basis vector. After this, we extended the result to all Hermite windows at the first canonical basis vector, and later to all Hermite windows at all elements in the finite span of the canonical basis vectors. Finally, we repeated all of the above steps in higher dimensions. Along the way, conclusions from the previous step were Lean verified and then utilized as inspiration for the LLMs to proceed to the next generalization.
\vspace{0.5em}

The objective of this subsection will be to present some of the key ideas that arose while completing the above steps. These ideas mainly occurred during the first few generalizations, specifically when passing to the first Hermite window at the first canonical basis vector and to the Gaussian window in higher dimensions. The final proof of Theorem~\ref{thm:local_Gen} is a technical modification of variants of these ideas, and we have not been able to find deeper meaning in it. For this reason, we have not presented the full proof of Theorem~\ref{thm:local_Gen} in detail, as it does not seem instructive. 
\vspace{0.5em}

With the above being said, it may simply be that we have not yet found the \emph{right} proof of Theorem~\ref{thm:local_Gen}, or perhaps even the right statement. Indeed, given that we now have a verification of correctness of Theorem~\ref{thm:local_Gen}, it is very natural to ask the following.

\begin{enumerate}
    \item Is there an analogue for the wavelet transform (where now the geometry is the upper half space) or on bounded domains (cf.~\cite{alaifari2019Stable,alaifari2025cheeger})?
    \item What is the true role of the Hermite windows; can the result be generalized to other windows which do STFT phase retrieval?
    \item What is the true role of the Gaussian measure; can the result be generalized to other (radial) log-concave measures?
    \item  There are dense sets where stability holds and dense sets where it does not. Can we quantify which set is bigger? In particular, how does the stability constant in Theorem~\ref{thm:local_Gen} depend on $P_0$, at least in the case of the Fock space? A related question is to quantify the ``size" of the set of pairs $(f,g)\in L^2(\mathbb{R}^d)\times L^2(\mathbb{R}^d)$ so that $V_gf$ does local stable phase retrieval in $L^2(\mathbb{R}^d)$.
\item Are there other applications of Theorem~\ref{thm:local_Gen}? Are there other perspectives on the Fock space (e.g., related to ancient solutions to heat flows, coherent states, etc.)~that could provide alternative methods to prove or extend our results?

\end{enumerate}
\subsubsection{A special case of Theorem \ref{thm:local_Gen}} We now discuss some of the ideas used to generalize  Theorem~\ref{main:local}. Importantly, we note that independent iterations of the same prompts resulted in different strategies to generalize the proof, which may be  a fact of independent interest.
\vspace{0.5em}

For simplicity, we discuss the special case of Theorem \ref{thm:local_Gen} where we take $d=k=1$ and $F_0=\Phi_{1,0}\coloneq \Phi_0=\overline{z}$. The general case is more technical, but follows a similar line of reasoning to what is explained below. In this special case, the result is the analogue of Theorem \ref{main:local} for the first true polyanalytic Fock space $\mathcal{H}_1(\mathbb{C})$, where $\Phi_0$ plays the role of the constant function. The proof proceeds by using the orthogonal reduction to reduce matters to proving the orthogonal coercivity bound 
\begin{equation}\label{orthcoercivityhermite}
\norm{U}_\gamma
\lesssim
\norm{\abs{\Phi_0 + U} - \abs{\Phi_0}}_\gamma
\qquad
\bigl(U \in \mathcal{H}_1 \cap \Phi_0^\perp\bigr).
\end{equation}
This is the analogue of Theorem \ref{thm:orthog}. Once we have this estimate, the overall architecture of the proof is the same as in the Fock space case, and with a bit of work, one can essentially reduce all of the estimates to small variations of the ones used in this case. However, there are two minor but important points when carrying out these reductions.
\begin{enumerate}
\item The circle estimates in Section \ref{subsec:circle} only apply to trigonometric polynomials supported on positive frequencies, and use the defect function $\rho(w)=||1+w|-1|$ and not $||\Phi_0+w|-|\Phi_0||$. Moreover, $U$ is now allowed to have a non-trivial zeroth order Fourier mode because $\Phi_0$ is supported at frequency $-1$. The fix to this is simple: after writing everything in polar coordinates, one simply factors (for $r>0$) out the basis vector $\Phi_0$ on both sides of the estimate \eqref{orthcoercivityhermite}. This shifts $U$ to positive frequencies and normalizes the defect function to the one used in the Fock case, where the circle estimates in Section \ref{subsec:circle} apply.
\item The second minor difference is in the annulus localization carried out in Section \ref{subsec:blocks}. Here, instead of the monomials $z^n$ (in the Fock case), we have to understand the mass concentration properties of the basis vectors $\Phi_n\coloneq \Phi_{1,n}$ to obtain the analogue of the bound \eqref{eq:monomial-bound}.
\end{enumerate}
Once these two points are addressed, the proof proceeds in an almost identical fashion to the Fock space case. The proof of the general result in Theorem \ref{thm:local_Gen} in one dimension is much more technical, but can ultimately be deduced by following a parallel line of reasoning. 

\subsubsection{Higher dimensional analogues} We now discuss higher dimensional analogues in $\mathbb{C}^d$. We mention here the simple analogue of the Fock space case $\Fd = \mathcal{F}^2(\C^d)$.

\begin{theorem}[Local SPR at $\mathbf{1}$ in $\mathcal{F}_d$]\label{mainhigherd:local}
There exists a constant $M=M_d> 0$ such
that for any $F\in \mathcal{F}_d$ we have
\[
\inf_{|\lambda|=1}\lVert F - \lambda \rVert_{\mathcal{F}_d}
\;\le\;
M\,
\bigl\|\, \lvert F \rvert - 1 \,\bigr\|_{\mathcal{F}_d}.
\]
\end{theorem}
As in the one-dimensional case, Theorem~\ref{mainhigherd:local} can be reduced to proving the following orthogonal coercivity estimate.
\begin{theorem}[Orthogonal coercivity on $\C^d$]\label{thm:orthogd}
Fix $d \ge 1$.  There exists a constant $C_d > 0$ such that every
$U \in \Fd$ with $U(0) = 0$ satisfies
\[
\norm{U}_{\Fd}^2
\le
C_d^2\, \norm{\rhoFn \circ U}_{\Fd}^2.
\]
\end{theorem}
This is the analogue of Theorem \ref{thm:orthog} in the higher dimensional setting. As in the one-dimensional Fock space, we still have an orthonormal basis of monomials, except here they are of the form
\[
e_\alpha(z) := \frac{z^\alpha}{\sqrt{\alpha!}}
\]
where $\alpha\in\mathbb{N}^d$. As before, an important point in proving the orthogonal coercivity bound is in understanding the regions of space where $e_{\alpha}$ concentrates most of its mass. This allows one to carry out the analogue of the frequency localization in Section~\ref{subsec:blocks}. This part is relatively similar to the one-dimensional case. 
\vspace{0.5em}

The part of the proof that requires a relatively new idea is the analogue of the circle estimates in Section \ref{subsec:circle}. This is clarified by writing the $\|\cdot\|_{\mathcal{F}_d}$ norm in polar form
\begin{equation*}\label{eq:polard}
\norm{F}_{\Fd}^2
=
\frac{2}{(d-1)!}
\int_0^\infty
r^{2d-1} \ee^{-r^2}
\biggl(
\int_{\Sphere} \abs{F(r\omega)}^2\,\dd \sigma_d(\omega)
\biggr)\dd r,
\end{equation*}
and so, one needs analogues (which hold for functions on $S^{2d-1}$) of the following estimates:
\begin{enumerate}
\item the local circle estimate
$\norm{P}_{L^2(S^1)}^2 \le 144 L\, \norm{\rhoFn \circ P}_{L^2(S^1)}^2$
for trigonometric polynomials with $L$ positive frequencies;
\item the high-frequency estimate
$\norm{P}_{L^2(S^1)}^2 \le 32\, \norm{\rhoFn \circ P}_{L^2(S^1)}^2$
for bands $\{N,\dots,N+L-1\}$ with $1343L^2 \le N^2$.
\end{enumerate}
Fortunately, by a simple averaging argument, both of these one-dimensional estimates can be lifted to the higher dimensional sphere $S^{2d-1}$. We state below the lifted estimates. The analogues of the local estimates and high frequency estimates are given, respectively, as follows. 
\begin{proposition}[Lifted local estimate]\label{prop:lifted-local}
Let $E \subset \N_{\ge 1}$ be finite with $\abs{E} = L$, and let
\[
P(\omega) = \sum_{n \in E} G_n(\omega),
\]
where each $G_n$ is the restriction to $\Sphere$ of a homogeneous
holomorphic polynomial of degree $n$.  Then
\[
\norm{P}_{L^2(\Sphere)}^2
\le
144 L\, \norm{\rhoFn \circ P}_{L^2(\Sphere)}^2.
\]
\end{proposition}
\begin{proposition}[Lifted high-frequency estimate]\label{prop:lifted-high}
Let $N,L \ge 1$ satisfy $1343L^2 \le N^2$, and let
\[
P(\omega) = \sum_{m=0}^{L-1} G_{N+m}(\omega),
\]
where each $G_{N+m}$ is the restriction to $\Sphere$ of a homogeneous
holomorphic polynomial of degree $N+m$.  Then
\[
\norm{P}_{L^2(\Sphere)}^2
\le
32\, \norm{\rhoFn \circ P}_{L^2(\Sphere)}^2.
\]
\end{proposition}
With this in hand, the proof of Theorem \ref{mainhigherd:local} follows a similar path to the one-dimensional case (albeit with some added technical, but routine, details). To obtain Theorem \ref{thm:local_Gen}, one may combine the above ideas with the ingredients described above for the one-dimensional Hermite case. While these are the core additional ingredients, the execution is rather technical, which is why we have omitted it.

\subsection{Proof of Theorem~\ref{thm:FNTLogSob}}\label{subsec:FNTproof}

We now show how Theorem~\ref{thm:FNTLogSob} follows from
Theorem~\ref{main:local} and
\cite[Corollary~1.2]{DolbeaultEstebanFigalliFrankLoss2025}.  The argument
has three main ingredients. The first is an identity between the norm of the gradient of $|F|$ and $|F'|$,
which converts the log-Sobolev deficit for entire functions into a Gaussian
log-Sobolev deficit for the modulus~$\abs{F}$. The second is the dimension-free
 Gaussian log-Sobolev inequality due to
Dolbeault--Esteban--Figalli--Frank--Loss \cite{DolbeaultEstebanFigalliFrankLoss2025}. The final ingredient is the standard covariance property of the
Fock space, which will allow us to transfer the local stability at
the constant function provided by Theorem~\ref{main:local} into local stability at any other extremizer. 
 \vspace{0.5em}

Before moving on with the proof of Theorem \ref{thm:FNTLogSob}, we state some preliminary facts. 
For $\beta \in \C$, let
\[
G_\beta(z) \;:=\; \ee^{\beta z - \abs{\beta}^2/2},
\qquad z \in \C.
\]
Apart from being the class of extremizers to the logarithmic Sobolev inequality, these functions also constitute the family of normalized reproducing kernels for the Bargmann-Fock space. In particular, we have $G_\beta \in \mathcal{F}^2(\C)$ with $\norm{G_\beta}_{\mathcal{F}} = 1$,
and $\abs{G_\beta(z)}
= \ee^{\operatorname{Re}(\beta z) - \abs{\beta}^2/2}$.
We define the \emph{Bargmann shift}
$W_\beta : \mathcal{F}^2(\C) \to \mathcal{F}^2(\C)$ by
\[
(W_\beta H)(z)
\;:=\; \ee^{\beta z - \abs{\beta}^2/2}\, H(z - \overline{\beta}),
\qquad H \in \mathcal{F}^2(\C).
\]
When computing the Fock norm of $W_\beta H$, a change of variables $w = z - \overline{\beta}$ 
together with the identity
$\abs{z}^2 = \abs{w}^2 + 2\operatorname{Re}(\beta w) + \abs{\beta}^2$
makes the weight
$\ee^{2 \operatorname{Re}(\beta z) - \abs{\beta}^2 - \abs{z}^2}$
become $\ee^{-\abs{w}^2}$. Hence, $W_{\beta}$ is an isometry.  A direct
computation yields $W_\beta W_{-\beta} = I = W_{-\beta} W_\beta$, so $W_\beta$
is unitary on $\mathcal{F}^2(\C)$ with $W_\beta^{-1} = W_{-\beta}$.
Note also that $W_\beta\, \mathbf{1} = G_\beta$ and
$W_{-\beta}\, G_\beta = \mathbf{1}$.
 \vspace{0.5em}

Next, we claim that for every $F \in \mathcal{F}^2(\C)$ and every
$\beta \in \C$,
\begin{equation}\label{eq:fnt-covariance}
\norm{\, \abs{W_{-\beta} F} - 1 \,}_{\mathcal{F}}
\;=\;
\norm{\, \abs{F} - \abs{G_\beta} \,}_{\mathcal{F}}.
\end{equation}
Since $W_{-\beta}$ is unitary on $\mathcal{F}^2(\C)$ and
$\norm{\mathbf{1}}_{\mathcal{F}} = \norm{G_\beta}_{\mathcal{F}} = 1$,
expanding both sides reduces \eqref{eq:fnt-covariance} to the identity
\[
\bigl\langle\, \abs{W_{-\beta} F},\, \mathbf{1} \,\bigr\rangle_{\mathcal{F}}
\;=\;
\bigl\langle\, \abs{F},\, \abs{G_\beta} \,\bigr\rangle_{\mathcal{F}}.
\]
This latter identity follows from a direct change of variables. Indeed, by writing $$\abs{W_{-\beta} F(z)}
= \ee^{-\operatorname{Re}(\beta z) - \abs{\beta}^2/2}\,
\abs{F(z + \overline{\beta})}$$ and substituting $w = z + \overline{\beta}$,
the exponent
$-\operatorname{Re}(\beta z) - \abs{\beta}^2/2 - \abs{z}^2$
becomes $$\operatorname{Re}(\beta w) - \abs{\beta}^2/2 - \abs{w}^2
= \log \abs{G_\beta(w)} - \abs{w}^2.$$ Thus,
\[
\begin{aligned}
\frac{1}{\pi}\int_{\C} \abs{W_{-\beta} F(z)}\,
\ee^{-\abs{z}^2}\,\dd m(z)
&=
\frac{1}{\pi}\int_{\C} \abs{F(w)}\,
\abs{G_\beta(w)}\, \ee^{-\abs{w}^2}\,\dd m(w)
\\
&=
\bigl\langle\, \abs{F},\, \abs{G_\beta} \,\bigr\rangle_{\mathcal{F}}.
\end{aligned}
\]
We are now ready to prove Theorem \ref{thm:FNTLogSob}. 

\begin{proof}[Proof of Theorem~\ref{thm:FNTLogSob}]
Let $F \in \mathcal{F}^2(\C)$ with $\norm{F}_{\mathcal{F}} = 1$.
We may assume $\norm{F'}_{\mathcal{F}} < \infty$, since otherwise the
left-hand side of Theorem~\ref{thm:FNTLogSob} is infinite and there is nothing
to prove.
 \vspace{0.5em}

\noindent
\emph{Step 1: A gradient identity.}\;
A direct application of the Cauchy--Riemann equations
gives $\abs{\nabla F}^2 = 2 \abs{F'}^2$.  We also have, for
entire $F$, the standard identity
\begin{equation}\label{eq:fnt-nabla-mod}
\abs{\nabla \abs{F}}^2
\;=\; \abs{F'}^2 \qquad \text{a.e. on } \C.
\end{equation}
To see this,  write $F = u + \ii v$ with $u, v$ real-valued and apply the Cauchy--Riemann
equations $u_{x_1} = v_{x_2}$, $u_{x_2} = -v_{x_1}$.  On $\{F \neq 0\}$,
$\nabla \abs{F} = (u \nabla u + v \nabla v) / \abs{F}$, and a
short computation gives
\[
(u u_{x_1} + v v_{x_1})^2
+ (u u_{x_2} + v v_{x_2})^2
\;=\; (u^2 + v^2)\, (u_{x_1}^2 + v_{x_1}^2),
\]
so that
$\abs{\nabla \abs{F}}^2
= u_{x_1}^2 + v_{x_1}^2 = \abs{F'}^2$.
As the zero set of an entire function on $\mathbb{C}$ is discrete, \eqref{eq:fnt-nabla-mod} follows. In particular,
$\abs{F} \in H^1(d\gamma)$ with
\begin{equation}\label{eq:fnt-mod-energy}
\int_{\C} \abs{\nabla \abs{F}}^2\, \dd\gamma
\;=\; \norm{F'}_{\mathcal{F}}^2.
\end{equation}

\noindent
\emph{Step 2: The application of \cite[Corollary~1.2]{DolbeaultEstebanFigalliFrankLoss2025}.}\;
Let $\dd\widetilde{\gamma}(x) = \ee^{-\pi \abs{x}^2}\,\dd x$
on $\mathbb{R}^2$ denote the normalized 
Gaussian measure. We recall that
\cite[Corollary~1.2]{DolbeaultEstebanFigalliFrankLoss2025} states that, for
every $u \in H^1(d\widetilde{\gamma})$,
\begin{equation}\label{eq:DEFFL}
\int_{\mathbb{R}^2} \abs{\nabla u}^2\, \dd\widetilde{\gamma}
\,-\, \pi \int_{\mathbb{R}^2} u^2\, \ln\!\biggl(\frac{u^2}{\norm{u}_{L^2(\widetilde{\gamma})}^2}\biggr) \dd\widetilde{\gamma}
\;\ge\;
\frac{\beta_\star\, \pi}{2}
\inf_{b \in \mathbb{R}^2,\, a \in \mathbb{R}}
\int_{\mathbb{R}^2} \bigl(u - a\, \ee^{b\cdot x}\bigr)^2\, \dd\widetilde{\gamma},
\end{equation}
where $\beta_* > 0$ is an absolute constant. We now apply
\eqref{eq:DEFFL} to $u(x) := \abs{F}(\sqrt{\pi}\, x)$.  Since the change of 
variables $z = \sqrt{\pi}\, x$ sends $\tilde{\gamma}$ to $\gamma$, tracking
each term in \eqref{eq:DEFFL} and using \eqref{eq:fnt-mod-energy} for the gradient contribution, we obtain
\[
\norm{u}_{L^2(d\widetilde{\gamma})}
\;=\; \norm{F}_{\mathcal{F}} \;=\; 1,
\qquad
\int_{\mathbb{R}^2} \abs{\nabla u}^2\, \dd\widetilde{\gamma}
\;=\; \pi \int_{\C} \abs{\nabla \abs{F}}^2\, \dd\gamma
\;=\; \pi\, \norm{F'}_{\mathcal{F}}^2,
\]
\[
\pi \int_{\mathbb{R}^2} u^2 \ln u^2\, \dd\widetilde{\gamma}
\;=\; \pi \int_{\C} \abs{F}^2 \ln \abs{F}^2\, \dd\gamma.
\]
For the deficit term on the right-hand side of \eqref{eq:DEFFL}, we write
$b = (b_1, b_2) \in \mathbb{R}^2$ and set
$\beta := (b_1 - \ii b_2) / \sqrt{\pi} \in \C$, so that
$b\cdot x = \operatorname{Re}(\beta z)$. Then
\[
\int_{\mathbb{R}^2} \bigl(u - a\, \ee^{b\cdot x}\bigr)^2\, \dd\widetilde{\gamma}
\;=\;
\int_{\C} \bigl(\abs{F(z)} - a\, \ee^{\operatorname{Re}(\beta z)}\bigr)^2\, \dd\gamma(z).
\]
Substituting all of the above into \eqref{eq:DEFFL} and dividing by~$\pi$, we obtain
\begin{equation}\label{eq:fnt-DEFFL-applied}
\norm{F'}_{\mathcal{F}}^2
\,-\, \int_{\C} \abs{F}^2 \ln \abs{F}^2\, \dd\gamma
\;\ge\;
\frac{\beta_\star}{2}
\inf_{\beta \in \C,\, a \in \mathbb{R}}
\norm{\, \abs{F} - a\, \ee^{\operatorname{Re}(\beta\,\cdot\,)} \,}_{\mathcal{F}}^2.
\end{equation}

\noindent
\emph{Step 3: Rewriting the deficit term.}\;
Since $\ee^{\operatorname{Re}(\beta z)}
= \ee^{\abs{\beta}^2/2}\, \abs{G_\beta(z)}$ and
$\ee^{\abs{\beta}^2/2} \in (0, \infty)$, the substitution
$a \mapsto a\, \ee^{\abs{\beta}^2/2}$ shows that
\[
\inf_{\beta \in \C,\, a \in \mathbb{R}}
\norm{\, \abs{F} - a\, \ee^{\operatorname{Re}(\beta\,\cdot\,)} \,}_{\mathcal{F}}^2
\;=\;
\inf_{\beta \in \C,\, a \in \mathbb{R}}
\norm{\, \abs{F} - a\, \abs{G_\beta} \,}_{\mathcal{F}}^2.
\]
For fixed $\beta \in \C$, set $t_\beta := \langle \abs{F},\, \abs{G_\beta} \rangle_{\mathcal{F}}$.
By Cauchy--Schwarz, $t_\beta \in [0, 1]$, and the optimal $a \in \mathbb{R}$ is
$a = t_\beta$, with minimum value $1 - t_\beta^2$.  Since
$\norm{\, \abs{F} - \abs{G_\beta} \,}_{\mathcal{F}}^2 = 2(1 - t_\beta)$
and
$1 - t_\beta^2 = (1 - t_\beta)(1 + t_\beta) \ge \tfrac{1}{2}\cdot 2(1 - t_\beta)$,
we obtain
\begin{equation}\label{eq:fnt-c-opt}
\inf_{\beta \in \C,\, a \in \mathbb{R}}
\norm{\, \abs{F} - a\, \abs{G_\beta} \,}_{\mathcal{F}}^2
\;\ge\;
\tfrac{1}{2}\, \inf_{\beta \in \C}
\norm{\, \abs{F} - \abs{G_\beta} \,}_{\mathcal{F}}^2.
\end{equation}

\noindent
\emph{Step 4: The application of Theorem~\ref{main:local}.}\;
Fix any $\beta \in \C$ and set $H := W_{-\beta} F \in \mathcal{F}^2(\C)$.
By \eqref{eq:fnt-covariance},
$\norm{\, \abs{H} - 1 \,}_{\mathcal{F}}
= \norm{\, \abs{F} - \abs{G_\beta} \,}_{\mathcal{F}}$.
By unitarity of $W_{-\beta}$ and the identity $W_{-\beta} G_\beta = \mathbf{1}$, we have
\[
\norm{H - c}_{\mathcal{F}}
\;=\; \norm{W_{-\beta}(F - c G_\beta)}_{\mathcal{F}}
\;=\; \norm{F - c G_\beta}_{\mathcal{F}}
\qquad \text{for every } c \in \C.
\]
Theorem~\ref{main:local} applied to $H$ therefore yields
\[
\inf_{\abs{c}=1} \norm{F - c G_\beta}_{\mathcal{F}}
\;\le\; M\, \norm{\, \abs{F} - \abs{G_\beta} \,}_{\mathcal{F}}.
\]
Squaring and taking the infimum over $\beta \in \C$, we obtain
\begin{equation}\label{eq:fnt-mainlocal}
\inf_{\substack{\beta \in \C \\ \abs{c}=1}}
\norm{F - c G_\beta}_{\mathcal{F}}^2
\;\le\;
M^2\, \inf_{\beta \in \C}
\norm{\, \abs{F} - \abs{G_\beta} \,}_{\mathcal{F}}^2.
\end{equation}
Combining \eqref{eq:fnt-DEFFL-applied}, \eqref{eq:fnt-c-opt} and
\eqref{eq:fnt-mainlocal}, it follows that
\[
\norm{F'}_{\mathcal{F}}^2
\,-\, \int_{\C} \abs{F}^2 \ln \abs{F}^2\, \dd\gamma
\;\ge\;
\frac{\beta_\star}{4 M^2}
\inf_{\substack{\beta \in \C \\ \abs{c}=1}}
\norm{\, F - c\, \ee^{\beta z - \abs{\beta}^2/2} \,}_{\mathcal{F}}^2.
\]
This proves Theorem~\ref{thm:FNTLogSob}.
\end{proof}

\section{Discussions on the Lean verification}\label{App}
This section discusses the formalization of Theorem \ref{thm:local_Gen}. The file \lean{DimdPoly.lean}  provides a self-contained version of the main statement that was Lean verified\footnote{This is in the spirit of the comparator approach \cite{LeanComparator}
to safely verify Lean proofs. Indeed, the repository also contains an appropriate \lean{Challenge.lean} file to be verified.}. When writing these statements, our goal was to maximize intelligibility for mathematicians who are not well-versed in Lean, even if it sometimes resulted in \emph{less idiomatic} Lean code.  \vspace{0.5em}

The full Lean code is available at https://github.com/susannabertolini/PhaseRetrieval. We refer the reader to the READMEs for an explanation of the  structure of the repository. We emphasize that although the reviewer facing statements have been carefully curated by the authors with the objective of making them understandable to a broad audience, the rest of the Lean code is of poor quality and is simply there to verify correctness -- it is not meant to be digested by the reader. \vspace{0.5em}

This autoformalization, and the trust we placed in it, is possible only because of the enormous efforts by the Mathlib community \cite{mathlib2020} in carefully formalizing mathematical objects in Lean and ensuring that their meaning is what the mathematical community expects.  In particular, our formalization contains plenty of mathematical statements such as \lean{MeasureTheory.volume.withDensity} or \lean{Real.sqrt} whose construction is not detailed in this section, but which have been thoughtfully designed and very carefully reviewed by the Mathlib community. \vspace{0.5em}

We now attempt to give an explanation of the content of the \lean{DimdPoly.lean} file.  The
formalized theorem discussed here is 
Theorem~\ref{thm:local_Gen}.  \vspace{0.5em}  

\theoremstyle{plain}
\newtheorem*{theorem*}{Theorem}

Throughout the discussion, we fix a positive dimension $d$.  In the Lean file,
the dimension $d \in \mathbb N$ is declared once as a file-level variable, with the hypothesis (named \lean{hd}) that $0<d$:
\begin{leancode}
variable {d : ℕ} (hd : 0<d)
\end{leancode}

\subsection{The canonical basis}

We start by defining the one-dimensional basis. For $k,n \in \mathbb N$ and
$z\in \mathbb C$, we use
\begin{equation*}
    \Phi_{k,n}(z) = \frac{1}{\sqrt{n!k!}}
    \sum_{j=0}^{\min\{n,k\}}(-1)^j j!
    \binom{k}{j}\binom{n}{j} z^{n-j}\overline{z}^{k-j}.
\end{equation*}

In Lean, everything has a type. Types are a primitive notion, in the spirit of sets. For example, $0$ has type $\mathbb N$, expressed as $(0:\mathbb N)$. The function $f(x) := x^2$ has type $\mathbb R \to \mathbb R$. One can define new functions, that take various inputs of various types, as follows:
\begin{leancode}
def DefinitionName (variable₀ : Type₀) (variable₁ variable₂: Type₁) : OutputType :=
    Insert your definition
\end{leancode}
With this in mind, $\Phi_{k,n}(z)$ is expressed as a function \lean{HermitePoly} that takes two integers \lean{k n}, one complex number \lean{z} and returns a complex number, as follows
\begin{leancode}
def HermitePoly (k n : ℕ) (z : ℂ) : ℂ :=
  (Real.sqrt ((Nat.factorial n) * (Nat.factorial k)))⁻¹ *
    ∑ j ∈ Finset.range (min n k + 1),
      ((-1) ^ j) * (Nat.factorial j ) *
        (Nat.choose n j) * (Nat.choose k j) *
        z ^ (n - j) * (star z) ^ (k - j)
\end{leancode}
Here, {\lean{Finset.range N}} is the finite set $\{0,\ldots,N-1\}$.  Thus, 
{\lean{∑ j ∈  Finset.range (min n k + 1)}} 
expresses the sum over $j=0,\ldots,\min\{n,k\}$.
 \vspace{0.5em}

The $d$-dimensional Hermite polynomials are tensor products of the $\Phi_{k,n}$:
\[
\Phi_{\vec\kappa,\vec\alpha}(\vec z)
:=\prod_{q=0}^{d-1}\Phi_{\vec\kappa_q,\vec\alpha_q}(\vec z_q).
\]

    In Lean, we define points $\kappa \in \mathbb N^d$ as maps from \lean{Fin d} to $\mathbb N$, written as  {\lean{κ : Fin d → ℕ}}. Here, \mbox{\lean{Fin d}} is the finite type with elements
    $\{0,1,\ldots,d-1\}$. In other words, the vector $\vec z := (z_0,\ldots,z_{d-1})$ is defined as the map that sends $q \in \{0, \dots, d-1\}$ to $z_q$.
    In Lean, this is written as follows:
\begin{leancode}
def Φ (κ α : Fin d → ℕ) (z : Fin d → ℂ) : ℂ :=
  ∏ q : Fin d, HermitePoly (κ q) (α q) (z q)
\end{leancode}
The notation is meant to be read literally: {\lean{κ}} and {\lean{α }}are
multi-indices, and {\lean{z}} is a point of $\C^d$.  The notation {\lean{∏ q : Fin d, ...}} denotes the finite product over all coordinates.
 \vspace{0.5em}

We next define finite linear combinations of the basis functions, indexed by their coefficients.
Given $\vec\kappa\in\mathbb N^d$, and a finitely supported coefficient family $\{F({\vec \alpha})\}_{\vec \alpha\in \mathbb N^d}$, with $F({\vec \alpha})\in \mathbb C$, we define
\[
    P_F(\vec z)
    = \sum_{\vec\alpha\in \operatorname{Supp}F}
      F(\vec\alpha)\Phi_{\vec\kappa,\vec\alpha}(\vec z).
\]
The Lean definition is:
\begin{leancode}
def TrueHermitePoly (κ : Fin d → ℕ)
    (F : Finsupp (Fin d → ℕ) ℂ) : (Fin d → ℂ) → ℂ :=
  fun z ↦ F.sum (fun α Fα ↦ Fα * Φ κ α z)
\end{leancode}
Let \lean{F} be a finitely supported map, which we denote as {\lean{Finsupp}}. The expression {\lean{F.sum}} is the
finite sum over the support of \lean{F}.  The definition
{\lean{TrueHermitePoly κ F}} is a function $\C^d\to\C$. 
 \vspace{0.5em}


We define the set of all (finite) linear combinations of the basis vectors by $$\mathcal H_{\vec \kappa, fin}(\mathbb C^d) := \operatorname{span}\{\Phi_{\vec\kappa, \vec\alpha} :\vec \alpha \in \N^d\},$$ which in Lean is the \lean{range} of all the elements obtained with the construction above.

\begin{leancode}
def TrueHermitePolys (κ : Fin d → ℕ) :
    Set ((Fin d → ℂ) → ℂ) :=
  Set.range (TrueHermitePoly κ)
\end{leancode}

\subsection{The Gaussian measure}

The Gaussian measure has density
\[
\frac{1}{\pi^d}e^{-\sum_{q=0}^{d-1}\abs{\vec z_q}^2}
\]
In Lean, we write
\begin{leancode}
def GaussianDensity (z : Fin d → ℂ) : ℝ :=
  (1 / Real.pi ^ d) *
    Real.exp (-Finset.sum Finset.univ fun q : Fin d ↦ ‖z q‖ ^ 2)
\end{leancode}
From this Gaussian density, which is a function, we can define a measure. In Lean, measures are not built from densities taking values in the reals, but from densities taking values in the non-negative reals and possibly $+\infty$, which are denoted by \lean{ENNReal}. With this in mind, one builds the Gaussian measure as follows:
\begin{leancode}
def γ : MeasureTheory.Measure (Fin d → ℂ) :=
  MeasureTheory.volume.withDensity fun z ↦ ENNReal.ofReal (GaussianDensity z)
\end{leancode}
We can now define the complex-valued $L^2(d\gamma)$ as \lean{(MeasureTheory.Lp ℂ 2 (γ (d := d)))}. In this case, Lean was not able to infer the value of \lean{d} automatically from the measure $d\gamma$ and it must be provided by us. We define \lean{TrueHermiteClosure κ} as $\mathcal H_{\vec \kappa} (\mathbb C^d) := \overline{\mathcal H_{\vec \kappa, fin}(\mathbb C^d)}$, which is the closure of the set of $L^2$ functions that are $\gamma$-almost-everywhere equal to an element in \lean{TrueHermitePolys}.

\begin{leancode}
def TrueHermiteFunctions (κ : Fin d -> ℕ) :
    Set (MeasureTheory.Lp ℂ 2 (γ (d := d))) :=
  closure { f | ∃ P ∈ TrueHermitePolys κ, f =ᵐ[γ] P }
\end{leancode}
In Lean, $L^p$ functions are defined as equivalence classes of functions which are a priori known to be $L^p$-integrable. Therefore, we can not directly take the closure of \lean{TrueHermitePolys}.

\subsection{Statement of the main result}
The theorem in \lean{DimdPoly.lean} is the following
stability statement:
\begin{theorem*}
Let $d > 0$ and $\vec\kappa\in\N^d$. For every $P_0\in \mathcal H_{\vec\kappa, fin}(\C^d)$ there exists a
constant $M=M(\vec\kappa,P_0)>0$ such that for every
$F\in\mathcal H_{\vec\kappa}(\C^d)$ we have
\[
\inf_{|\theta|=1}\lVert P_0-\theta F \rVert_{\gamma_d}^2
\le
M^2\,\bigl\|\, \lvert P_0\rvert-\lvert F\rvert\,\bigr\|_{\gamma_d}^2.
\]
\end{theorem*}

In order to state this theorem in Lean, we rewrite it in several key ways. First, in the statement of the theorem, we must separate the hypothesis that $P_0$ is a function \lean{(Fin d -> ℂ) -> ℂ} from the statement that $P_0$ \emph{belongs} to \lean{TrueHermitePolys κ}. Second, in order to avoid the heavy set/inf notation, we prove the equivalent statement  
\[
\exists \  \theta \in \mathbb C\text{ such that } |\theta|=1 \text{ and } \lVert P_0 -\theta F\rVert_{\gamma_d}^2
\le
M^2\,\bigl\|\, \lvert P_0\rvert-\lvert F\rvert\,\bigr\|_{\gamma_d}^2.
\]
With these modifications, the Lean statement of the stability theorem reads as follows:

\begin{leancode}
theorem stable_phase_retrieval
    (hd : 0 < d) (κ : Fin d -> ℕ)
    {P : (Fin d -> ℂ) -> ℂ)} (hP : P ∈ TrueHermitePolys κ) :
    ∃ M : ℝ, 0 < M ∧
      ∀ Q ∈ TrueHermiteFunctions κ,
        ∃ θ : ℂ, ‖θ‖ = 1 ∧
          ∫ z, ‖P z - θ * Q z‖ ^ 2 ∂ γ
            ≤ M ^ 2 * ∫ z, (‖P z‖ - ‖Q z‖) ^ 2 ∂ γ := by
\end{leancode}
\begin{remark}
    The Lean statement utilizes $\theta$ instead of $\lambda$. The discrepancy in notation is due to the fact that, in Lean, the constant $\lambda$ is reserved notation.
\end{remark}
The reader may explore the statements (but not the proof) in an interactive Lean~4 environment on \href{https://live.lean-lang.org/\#codez=JYWwDg9gTgLgBAWQIYwBYBtgCMBQOJgCmAdnAM4DGBhAJnAELADmA8kVCtGXsRMVeACuMJFnSFyhCjGB8eSEITJgkFCQDEIFANYBlAAoAlAHQARUDX0R0ATzwB6ALSO4AFVQTeUEEkwAvWjg\%2BQkcANyQoYFFxOAAJQm9gGBDNHTgsJDJgMjhIW14QKPRjOEd7HBpCADM4hMLkq1s4AAptOFIALjhAVEIASha\%2FOC7AIEJ\%2BkaGAXhw4Fubmw0JfYzIAR1hZgDkUYyrVGGgi9qG4QFxCfoAqFq2YHb2D3zg2rrPesbhRwG8CAE6L6Zm4QCIRHAAFZwQAQRHB1MBiGRCDcOMQmBJmoVSKQ2gBqOAARl6ABo\%2Fv9Zo5scdRnAAHogi5Xba7aT3dAgsm\%2FIlE5rXYwUVAQCCwo6gkY0jnbbm8\%2FltQXvVlsmaDKnNUguYHCsgiKBwPz9BVtZW9PDhSLRCQAbzoXW6AF8HM43B44AADGgOsIRKJiCQZLI5QjiRTEG6lcqVGqAMuAWoAu4DggEbgY5Q0h0RwAPh6\%2FWagy68bgiZT5PGHSmM0A8ERwFZx6HZvG1RINaw2SOltOxlZpwYtm0udwSKrQpISeI1lJaNp5GwFQ4oETcwL7OBIOA94h9uBUao9ijAEjwCIcGwlMoVapuKCCQgD\%2BqERr15pRzMVnOpwktdTlmGCMBgZ\%2F30opvrS47NFmD7ksm7yTH8VSCKQgwTCm6jLIIIALlBMYrnAsFoZc4ZRrGfgdnaEiwvAEA1L4TKLsu559o4qQjnW46\%2BDkKBwOIoS\%2Bo6EYOvuwZHq4J5nnUfZXjkN6vtmP6pkMT66HCsxARJIG5mMhZwDJ8JIIiyJ8aeVG1k0Eb6jgTidvaUCLOgrroKedAAOJIIIZBZBp2YkFkMD1nwjryaBwxcUGh41Ew9mOVExCmK5STXhmkLfj5bwnOBMzNKS9hwAsSxgMAlLZjS6XFIQAAen7NDR0JEQhSHxuVUHAKEyGkGWd4JuhKaAGgEbatdlABMhnGQRuRQBANCCNIcB2Q5TmkIomSCGZQSkA63m5n5B4hnAgDNwMcCCLGQs2EF20B7ttM1zYBsVKYliA7XtB1QHuoTWIhhDGAA7kkqDhTCkX1ZqLVwAAohsGx5cYJF5S0QUTaFn1ufWWr4V2K7oHygTjSFzkOgAMgATQ6SOZDkSLEAkKCBFg9YUck1YXjRw65PRECFIx3EBceOmCck6hQdIsgwg2TXiaBf4dNJsnNMdu1mbde6Y5\%2BwxwF1LSbc05oTDlryXRQyOS6aC5wAAPnAgDARHA\%2BjgmzAmDsJcARlWNQTIACrgANrrQAuqbcDWkZtq6CIHq5KgmQSGZMCRIQ4TkdAY3BZNjo43jmD1DkJELr2VO6UOaSjgx6BkCzaCENAhBIWqxoAPpgIHsJlyHYcR0\%2BzSoOacAAAxwAAPDl\%2FMxc1Qv6uyZtdGdvcXXFLSoIPHsQtplsXtbBlSUSJsIMcJxVm3ncr4A5ERPjMgAARHAACK5szxnXP8DIfA5Lbu\%2F\%2FCbgAdwGSVatffnVq6SO\%2BykSgDURJqz9m4MFwj9LjHz8J1KkitABARBtW\%2BspAAmRIgbqcBLi\%2Fz8FWZorUAGdRcK1UBrVtQKzgNAzaBZ0h2BmGQaAd08AkDoLRAwJhzAgEsHWIAA}{this live.lean-lang.org link}. The service live.lean-lang.org is provided by the Lean Focused Research Organization.

\bibliographystyle{plain}
\bibliography{sources}
\end{document}